\documentclass{article}

\usepackage{amssymb,amsmath,amsthm}
\usepackage[utf8]{inputenc}
\usepackage{xcolor}
\usepackage{graphicx,tikz,subfigure,float}
\usepackage{enumitem}
\usepackage[section]{algorithm}
\usepackage{algpseudocode}
\usepackage{indentfirst}
\usepackage{authblk}

\numberwithin{equation}{section}

\newtheorem{theorem}{Theorem}

\newtheorem{lemma}[theorem]{Lemma}
\newtheorem{definition}[theorem]{Definition}

\newtheorem{openquestion}[theorem]{Open question}

\newcommand{\density}{\alpha}
\newcommand{\bb}[1]{\mathrm{#1}}
\newcommand{\Image}{\bb{Im}}
\newcommand{\dom}{\bb{dom}}
\newcommand{\DB}{\texttt{DB}}
\newcommand{\IF}{\texttt{IF}}
\newcommand{\SB}{\texttt{SB}}

\newcommand{\n}{n}
\newcommand{\cars}{C}
\newcommand{\goats}{G}
\newcommand{\subblock}{\sqsubseteq}
\newcommand{\oc}{\texttt{sparseSB0}}
\newcommand{\outcome}{\texttt{sparseSB}}
\newcommand{\sparse}{\texttt{sparse}}

\newcommand{\floor}[1]{\lfloor #1 \rfloor}
\newcommand{\ceil}[1]{\lceil #1 \rceil}
\newcommand{\NULL}{\texttt{NULL}}
\newcommand{\A}{A}

\newcommand{\f}{f}
\newcommand{\g}{g}
\newcommand{\F}{F}
\newcommand{\B}{B}
\newcommand{\G}{G}
\newcommand{\Pos}{P}
\newcommand{\h}{h}
\newcommand{\prin}{p}
\newcommand{\T}{\bb{T}}

\newcommand{\con}{\texttt{constructDB}}
\newcommand{\cgone}{\texttt{updateCG1}}
\newcommand{\cgr}{\texttt{updateCG}}
\newcommand{\cg}{\texttt{updateCGadaptive}}

\title{Comparing disorder and adaptability in stochasticity\thanks{Liling Ko was partially supported by Caroline Terry's NSF grant DMS-2115518.}\thanks{Collaborations were supported by the Mathematics Research Institute of The Ohio State University Mathematics Department and the Dartmouth Shapiro Visitors Fund.}}

\author[1]{Liling Ko}
\author[2]{Justin Miller}
\affil[1]{lko.390@osu.edu, Ohio State University}
\affil[2]{justin.d.miller@dartmouth.edu, Dartmouth College}

\begin{document}

\maketitle

\textbf{Abstract:} In the literature, there are various notions of stochasticity which measure how well an algorithmically random set satisfies the law of large numbers. Such notions can be categorized by disorder and adaptability: adaptive strategies may use information observed about the set when deciding how to act, and disorderly strategies may act out of order. In the disorderly setting, adaptive strategies are more powerful than non-adaptive ones. In the adaptive setting, Merkle \cite{merkle_2003} showed that disorderly strategies are more powerful than orderly ones. This leaves open the question of how disorderly, non-adaptive strategies compare to orderly, adaptive strategies, as well as how both relate to orderly, non-adaptive strategies. In this paper, we show that orderly, adaptive strategies and disorderly, non-adaptive strategies are both strictly more powerful than orderly, non-adaptive strategies. Using the techniques developed to prove this, we also make progress towards the former question by introducing a notion of orderly, ``weakly adaptable'' strategies which we prove is incomparable with disorderly, non-adaptive strategies. This is the first known demonstration of the fact that the considered stochasticity notions do not form a chain.

\section{Introduction}

\subsection{The landscape of stochasticity}

Stochasticity can be thought of as a version of the Monty Hall problem played in an infinite and computable setting: in this game, countably many doors $d\in\omega$ are arranged in a row. Behind each door $d$, the host $\A\subseteq\omega$ hides a car ($\A(d)=1$) or a goat ($\A(d)=0$). Infinitely many of the doors must hide cars ($|\A|=\infty$). The host claims that the assignment has been done independently for each door with probability $\density$ of the door hiding a car and probability $1-\density$ of the door hiding a goat. Contestants $f$ are tasked with demonstrating the host wrong by exhibiting a computable strategy which selects an infinite subsequence of doors in which the probability of a door hiding a car is not $\density$, which is possible almost never for an assignment of cars and goats which is truly independently random with probability $\density$. Changing the ruleset the contestants must follow changes the stochasticity notion.
\\
\par
Previously studied notions of stochasticity include von Mises-Wald-Church (MWC) stochasticity, Church stochasticity, Kolmogorov-Loveland (KL) stochasticity, and injection stochasticity (which is equivalent to the notion of intrinsic density introduced by Astor \cite{intrinsicdensity}).
\\
\par
In Church stochasticity, the contestants may use any total computable strategy which chooses a subsequence of doors by opening every door in a row according to the original ordering. Before opening the next door, the contestant is allowed to take or leave the object behind it. This decision is made (computably) with knowledge of what was behind the previously opened doors. If the host uses $A$ to assign goats and cars and challenges the contestants to disprove that they did so randomly with probability $\density$, then $A$ is said to be $\density$-Church stochastic if the host wins against all such strategies. Church stochasticity is \emph{orderly} because contestants must open all doors in order, and is \emph{adaptive} because contestants are permitted to use the information gained from previously opened doors in deciding whether to take the object behind the next one. 
\\
\par
The only modification between Church stochasticity and MWC stochasticity is that the contestants may use partial computable strategies rather than just total computable strategies. This notion is also orderly and adaptive, and at first glance appears quite similar to Church stochasticity. However, Ambos-Spies \cite{Ambosspies} showed that there are computable randoms which are not MWC stochastic, which proves that these notions do not coincide since every computable random is Church stochastic. (See Downey-Hirschfeldt \cite{Downey2010} Chapter 7.4 for a review of Church and MWC stochasticity.)
\\
\par
In KL stochasticity, contestants no longer need to respect the original ordering. They may visit doors out of order, and may still make decisions using information obtained from doors previously opened. This notion is disorderly and adaptive. Merkle \cite{merkle_2003} proved that KL stochasticity is strictly stronger than MWC stochasticity. (See also Merkle et al. \cite{merkle2006kolmogorov}.)
\\
\par
Astor \cite{intrinsicdensity} originally developed the notion of intrinsic density in the context of approximate computability, but proved that it is equivalent to the notion of injection stochasticity: here contestants are allowed to select doors out of order, but may not open any of them. They simply choose a comptable selection process to select their subsequence, and then all doors are opened simultaneously. This notion is disorderly and non-adaptive. If $A$ and $B$ are both $\density$-injection stochastic, then so is $A\oplus B$ \cite{thesis}, which immediately shows that injection stochasticity implies none of the previously mentioned stochasticity notions: with the ability to open doors without taking the hidden object, we can open even doors and only open odd doors if the object hidden behind the corresponding even door is a car.
\\
\par
It remains open whether MWC stochasticity implies injection stochasticity/ intrinsic density. Also, note that the above list does not include an orderly, non-adaptive notion of stochasticity. We shall give a formal definition of an orderly, non-adaptive notion of stochasticity and prove that it is strictly weaker than all of the above notions. The techniques from this proof shall allow us to make progress on the former question by introducing a new orderly, adaptive notion of stochasticity which is incomparable to intrinsic density.

\subsection{Definitions}

\begin{definition}
    Let $A\subseteq\omega$.
    \begin{itemize}
        \item The density of $A$ at $n$ is $\rho_n(A)=\frac{|A\upharpoonright n|}{n}$, where $A\upharpoonright n=A\cap\{0,1,\dots,n-1\}$.  
        \item The upper density of $A$ is $\overline{\rho}(A)=\limsup_{n\to\infty} \rho_n(A)$. 
        \item The lower density of $A$ is $\underline{\rho}(A)=\liminf_{n\to\infty} \rho_n(A)$. 
        \item If $\overline{\rho}(A)=\underline{\rho}(A)=\alpha$, we call $\alpha$ the (asymptotic) \emph{density} of $A$ and denote it by $\rho(A)$.
    \end{itemize}
\end{definition}

Asymptotic density is the main tool for analyzing whether something ``satisfies'' the law of large numbers. For example, if the set $A$ is generated by countably many independent $p$-Bernoulli random variables to determine if $n\in A$ for each $n$, then the law of large numbers says that $\rho(A)=p$ with probability $1$. Stochasticity notions require that subsequences selected according to specific rules also have density $p$.

\begin{definition}
    Let $\A\subseteq\omega$.
    \begin{itemize}
        \item The computable upper density of $\A$ is
        \[\overline{R}(\A) :=\sup\{\overline{\rho}(f^{-1}(A)):\; f\text{ total, computable, and increasing}\}.\]
        Computable lower density $\underline{R}(A)$ is defined analogously with $\inf$.
        \item The \emph{computable density} of $\A$ is
        \[R(\A) :=\underline{R}(\A), \text{ if } \overline{R}(\A)=\underline{R}(\A).\]
        $A$ is computably small if $R(A)=0$.
    \end{itemize}
\end{definition}

Alternatively, this is known as \emph{increasing stochasticity}. Replacing the increasing requirement by merely injective yields injection stochasticity, also known as intrinsic density. It is then immediate that $A$ having intrinsic density $\density$ implies $R(A)=\density$. Using the notation and techniques from \cite{thesis} and \cite{stochasticity}, it is straightforward to prove that this behaves the same under joins, {\tt into}, and {\tt within} as intrinsic density.

\begin{definition}
Let $A\subseteq\omega$.
\begin{itemize}
    \item A skip sequence $\sigma$ is an element of $(\omega\times 2)^{<\omega}$ such that the first coordinates of $\sigma$ are strictly increasing. $\sigma(k)_0$ denotes the first coordinate of the $k$-th entry in $\sigma$, and $\sigma(k)_1$ denotes the second coordinate. Let $S$ be the set of all skip sequences. 
    \item A skip rule is a total function $f:S\to\omega$ such that $f(\sigma)>\sigma(|\sigma|-1)_0$ for all $\sigma\in S$.
    \item The characteristic function of $f(A)$ is defined recursively via:
    \begin{itemize}
        \item $\chi_{f(A)}(0)=A(f(\emptyset))$
        \item $\chi_{f(A)}(n+1)=A(f(f(A)\upharpoonright n+1))$
    \end{itemize}
    \item $A$ is $\density$-weakly stochastic if 
    \[\rho(f(A))=r\]
    for all computable skip rules $f$.
\end{itemize}
\end{definition}

The first coordinate in each ordered pair of a skip sequence represents the door selected, and the second coordinate represents what was behind the door. Skip rules are functions which can use a skip sequence (i.e. all previously observed information) to decide how to proceed.
\\
\par
We shall prove that weak stochasticity is incomparable with intrinsic density, and that both imply but are not equivalent to computable density. We shall also show that weak stochasticity is not equivalent to Church or MWC stochasticity despite being a notion which is orderly and adaptive. This is the first known proof that the considered stochasticity notions do not form a chain. This will give us the following ``zoo'' (Figure \ref{fig:zoo}) of stochasticity notions, which combines our work with the existing landscape described above:

\begin{figure}[H]
    \centering
    \includegraphics[width=0.55\textwidth]{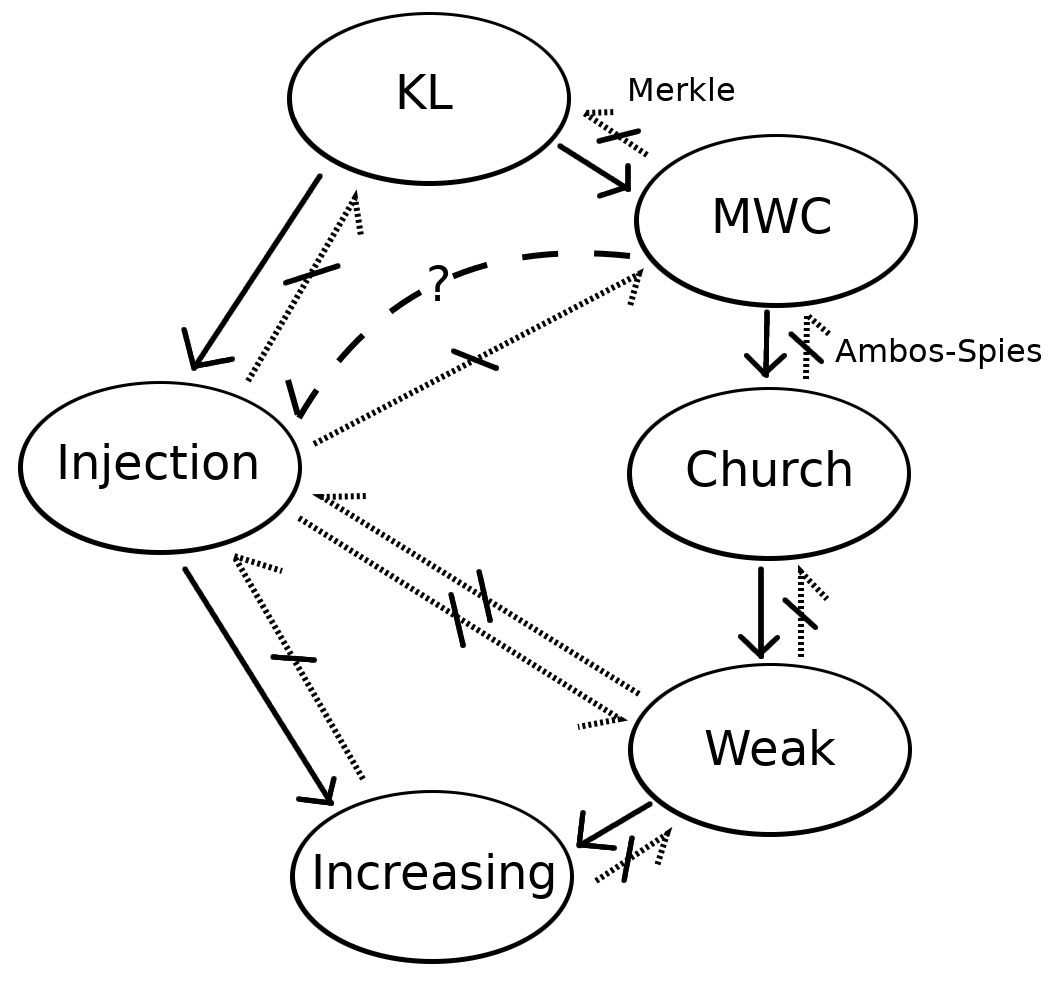}
    \caption{The dependencies between various notions of stochasticity are summarized here. The solid lines are implications which are immediate from definitions. The dotted lines are non-implications which have been proven, and the dashed line is open. Church's relationship to Injection is in the same state as MWC: one direction is known, the other is open. The referenced citations are \cite{Ambosspies} and \cite{merkle_2003}. All uncited non-arrows are either found in this paper or immediate from the work in \cite{thesis}.}
    \label{fig:zoo}
\end{figure}

\section{Computable density does not imply intrinsic density} \label{sec:disorder-stronger-order}

We shall first prove that computable density does not imply intrinsic density. We do so for density $0$ in Theorem~\ref{thm:disorder-stronger-order}, but we shall extend this to the entire unit interval in Subsection \ref{sec:alpha}.

\begin{theorem} \label{thm:disorder-stronger-order}
    A host $\A\subset\omega$ that beats all orderly contestants $\f\leq_{\mathrm{T}}\emptyset$ may not beat all disorderly ones $\h\leq_{\mathrm{T}}\emptyset$. That is, there is a set $A$ which is computably small but not intrinsically small. Furthermore, the constructed $A$ is not Church stochastic.
\end{theorem}

Intuitively, this theorem shows that in the non-adaptive setting, disorderly strategies outperform orderly strategies. Our opponents are the orderly contestants
\begin{align}
    \f:\; &\omega\to\omega,\\
    \forall(t) &[\f(t+1)>\f(t)].
\end{align}
Contestant $\f$ can be thought of as a map from time-stamp $t$ to door index $\f(t)$. We can exhaust our computable opponents by fixing an effective enumeration of the partial computable increasing functions:
\[\f_0,\f_1,\ldots.\]

To construct host $\A$ witnessing the Theorem~\ref{thm:disorder-stronger-order}, we will use $\emptyset'$ as an oracle:
\begin{lemma} \label{lemma:below-jump}
    We can ensure in Theorem~\ref{thm:disorder-stronger-order} that $\A\leq\emptyset'$.
\end{lemma}

Besides $\A$, we also need to construct a disorderly contestant $\h:\omega\to\omega$, again a map from time-stamp to door index. To ensure $\h$ receives enough cars, we shall satisfy for all $s\in\omega$ the global requirement
\begin{align} \tag{$\G_s$} \label{eq:G}
    \G_s:&\; (\exists n>s)\; \left[\left|[0,n(s+1)) \cap \h^{-1}(\A)\right|\geq ns\right],
\end{align}
which says that during the time interval $[0,n(s+1))$, $\h$ receives cars from $\A$ at least $ns$ many times. Therefore, the proportion of times that $\h$ receives cars is
\[\rho_{n(s+1)}(h^{-1}(A\restriction n(s+1))) \geq \frac{ns}{n(s+1)}=1-\frac{1}{s+1}.\]
Since $s$ can be arbitrarily large, we get
\[\bar{\rho}(h^{-1}(A)) = 1>0,\]
as required for $\A$ to be not intrinsically small.

To ensure opponent $\f=\f_e$ rarely receives cars, $\A$ must also be constructed to satisfy positive requirements
\begin{align} \tag{$\Pos_\f$} \label{eq:Pf}
    \Pos_\f: \f \text{ is total } \Rightarrow (\forall s)\; \Pos_{\f,s},
\end{align}
where $\Pos_{\f,s}$ is the sub-requirement
\begin{align} \tag{$\Pos_{\f,s}$} \label{eq:Pfs}
    \Pos_{\f,s}: (\exists t')(\forall t>t') \left\{|[\A(t),\A(t+1)]\cap \Image(\f)| \geq s\right\}.
\end{align}
We often drop subscripts when context is clear. If $\f$ is total, $\Pos_{\f,s}$ says that eventually, which is after time-stamp $t'$, when $\A$ is looking for a new door $\A(t+1)$ to store their next car, they will avoid at least the next $s$ many doors opened by $\f$. Since $\f$ never returns to smaller doors, by waiting for $\f$ to pick enough doors exceeding $\A(t)$, $\Pos_{\f,s}$ will be satisfied, and hence
\[\bar{\rho}(A\cap\Image(f)) \leq \frac{1}{s}.\]
If all $\Pos_{\f,s}$ are satisfied, $\bar{\rho}(\A\cap\Image(\f))=0$, as required for $\A$ to be computably small.

\subsection{Beating a single $\f$ using disordered blocks ($\DB$s)} \label{sec:1}
Consider a basic construction module where $\A$ and $\h$ only need to beat one opponent $\f$.\\

A natural attempt is to let $\h$ get to large doors quickly, or in other words, for $\h$ to be a fast growing function at first. We refer to this behaviour as $\h$ opening doors \emph{sparsely}. The hope is for $\f$ to be less quick, so that $f$ opens many doors within every \emph{gap} interval $(\h(t),\h(t+1))$ skipped by $\h$. We refer to such an outcome as $\f$ opening doors \emph{densely}, as sketched on the left of Figure~\ref{fig:outcomes-1}. Host $\A$'s \emph{strategy} against this outcome is to put the cars only at $\h$'s sparsely spaced doors. With sufficient $\f$-density, opponent $\f$ will open more than $s$ many goat-filled doors at the gaps for every possible car received at each sparsely spaced door, thereby satisfying $\Pos_{\f,s}$. If $\h$ continues opening doors sparsely for long enough, the length of time over which it receives cars will be sufficient to satisfy $\G_s$.\\

But what if $\f$ opens doors sparsely, just like in $\h$'s initial behaviour? We say that $\f$ has a \emph{sparse outcome}. By copying $\h$, contestant $\f$ will receive the same cars as $\h$ if we kept to the earlier strategy, which is too many cars for $\Pos_{\f,s}$. This is where $\h$'s advantage of being disorderly comes in: We pause $\h$'s increasing behaviour and force $\h$ to return to the gaps from the doors skipped earlier; $\f$ can no longer copy this behaviour. $\A$'s strategy will be to place cars only at a gap that $\f$ is sparse within. Refer to the right of Figure~\ref{fig:outcomes-1} for a sketch. By making the gaps wide, the time-interval within which $\h$ receives cars can be long enough to satisfy $\G_s$. But $\f$ might also receive cars at this gap. To ensure $\Pos_{\f,s}$ is satisfied, we remove all the cars from the doors opened by $\f$. The number removed is bounded from the sparseness of $\f$, allowing $\G_s$ to remain satisfied.\\

Thus, $\h$ alternates between being a fast growing function for a brief moment, then a gap-filling one for extended periods. We contain this disorderly behaviour into what we call a \emph{disordered block} ($\DB$). Abusing some notation, $\h$ is the concatenation of these blocks (Figure~\ref{fig:h}):
\begin{align} \label{eq:concat-h}
    \h=\DB_0 \;^\frown \DB_1 \;^\frown \ldots,
\end{align}
where $\DB_s$ refers to the $s$-th disordered block. Each $\DB$ is a bijection from a time-interval to a door-interval that continues from the previous block:
\begin{align*}
    \DB &:[a,a'] \to [a,a'] \text{ is a bijection from time-stamp to door index},\\
    \DB &\text{ often also refers to } \bb{Im}(\DB), \text{ its set of doors},\\
    \floor{\DB} &:= \min(\DB), \text{ the first door},\\
    \ceil{\DB} &:= \max(\DB), \text{ the last door},\\
    |\DB| &:=|\bb{Im}(\DB)|, \text{ the number of doors},\\
    \DB_0 &:\emptyset\to\emptyset,\\
    \floor{\DB_{s+1}} &=\ceil{\DB_s}+1.
\end{align*}

\begin{figure}[h]
    \centering
    \includegraphics[width=0.3\textwidth]{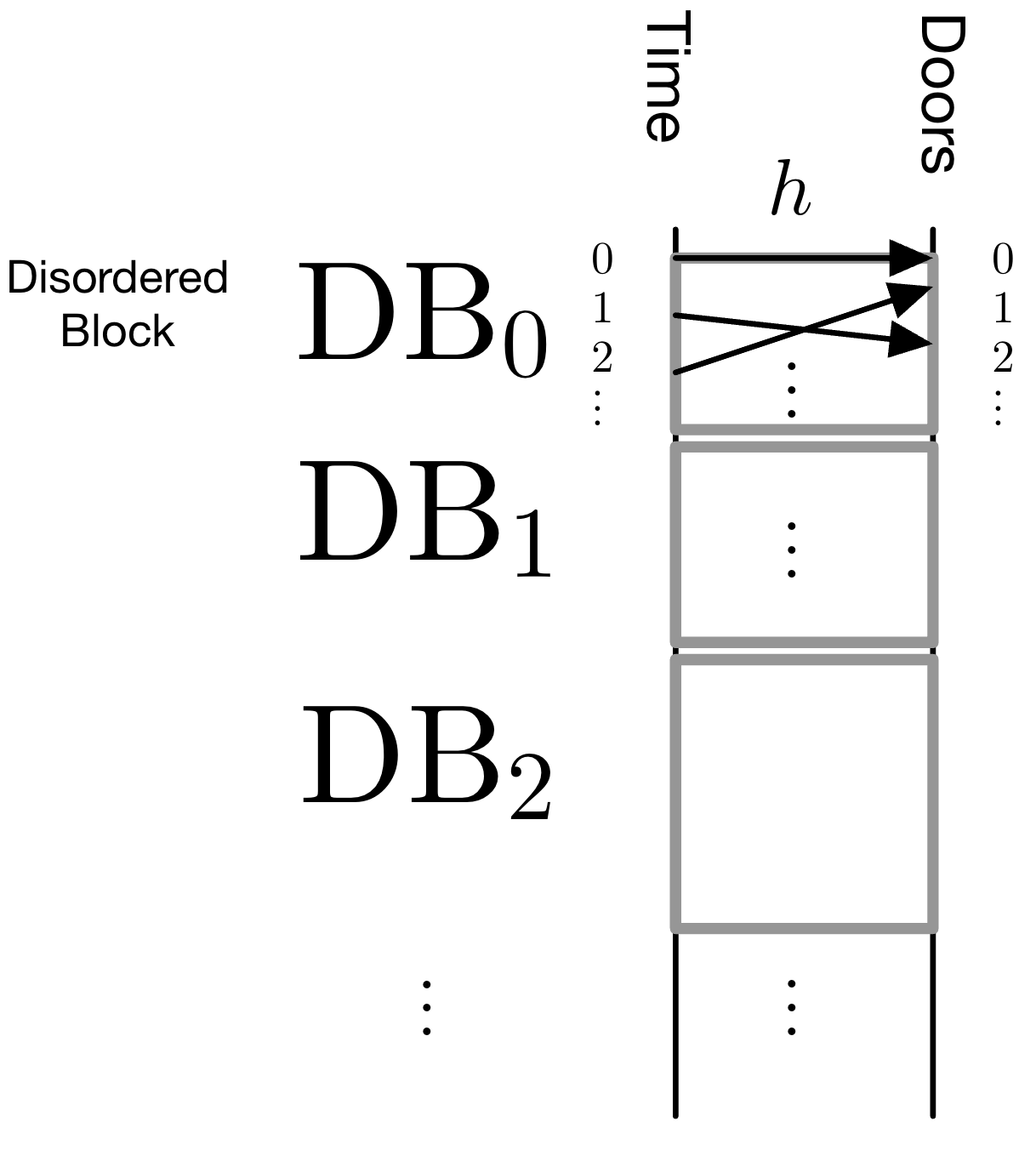}
    \caption{$\h$ is constructed as a concatenation of \emph{disordered blocks} $\DB$, each of which is a bijection viewed as a map from time-stamp to door index.}
    \label{fig:h}
\end{figure}

\subsubsection{$\DB$ structure} \label{sec:DB}
From the overview, each $\DB$ is a bijection from time to doors that begins as an increasing function, then becomes a sequence of functions that fill the gaps between the increasing doors (Figure~\ref{fig:DB}). We refer to the increasing part as the \emph{increasing function $\IF$} of $\DB$, written as $\IF\ll\DB$:
\begin{align*}
    \IF\ll\DB &\text{ is an increasing function from time-stamp to door index,}\\
    \IF &\text{ often also refers to } \bb{Im}(\IF), \text{ its set of doors},\\
    \bb{dom}(\IF) &\preceq \bb{dom}(\DB),\\
    \bb{Im}(\IF) &\subseteq \bb{Im}(\DB),\\
    \floor{\IF} &:=\min(\IF)=\floor{\DB}, \text{ the first door of } \DB,\\
    \ceil{\IF} &:=\max(\IF), \text{ the last door of } \IF,\\
    |\IF| &:=|\bb{Im}(\IF)|, \text{ the number of doors}.
\end{align*}

Also, we refer to the $i$-th gap-filling function as a \emph{sub-block $\SB_i$ of $\DB$}, written as $\SB_i\subblock\DB$:
\begin{align*}
    (\forall i<|\IF|)\; \SB_i &\subblock\DB \text{ is a bijection from time-stamp to door index},\\
    \SB &\text{ often also refers to } \Image(\SB), \text{ its set of doors},\\
    \Image(\SB_i) &\;=(\IF(\floor{\IF}+i),\IF(\floor{\IF}+i+1)), \text{ the } i\text{-th gap at } \IF,\\
    |\SB| &\;:=|\Image(\SB)| =|\dom(\SB)|, \text{ the number of doors},\\
    \min(\dom(\SB_0)) &\;=\max(\dom(\IF))+1, \text{ the first time stamp after } \IF,\\
    \min(\dom(\SB_{i+1})) &\;=\max(\dom(\SB_i))+1, \text{ the first time stamp after } \SB_i,\\
    \floor{\SB} &:=\min(\SB), \text{ the first door of } \SB,\\
    \ceil{\SB} &:=\max(\SB), \text{ the last door of } \SB.
\end{align*}

Thinking of the $\DB$ structure in terms of \emph{levels of block nestedness} will help us to generalize to future constructions against multiple opponents. To beat just one opponent, as we shall see, each $\SB$ above can be a translation function. As such, $\SB$ is considered its own increasing function and does not contain sub-blocks. Therefore, we think of these $\SB$'s as level 0 disordered blocks. On the other hand, the $\DB$ described earlier (Figure~\ref{fig:DB}) is considered to be at level $l=1$, since its sub-blocks are at level 0. We do not consider $\IF\ll\DB$ to be a sub-block of $\DB$. We shall see in Section~\ref{sec:construct-db2} and \ref{sec:construct-dbr} how higher level disordered blocks are used to fight more opponents.\\

Abusing more notation, we consider $\DB$ to be a concatenation of its increasing function followed by its sub-blocks:
\begin{equation} \label{eq:concat-db}
    \DB:= \;\IF^\frown \SB_0 \;^\frown \ldots^\frown \SB_{|\IF|}.
\end{equation}

\begin{figure}[h]
    \centering
    \includegraphics[width=0.5\textwidth]{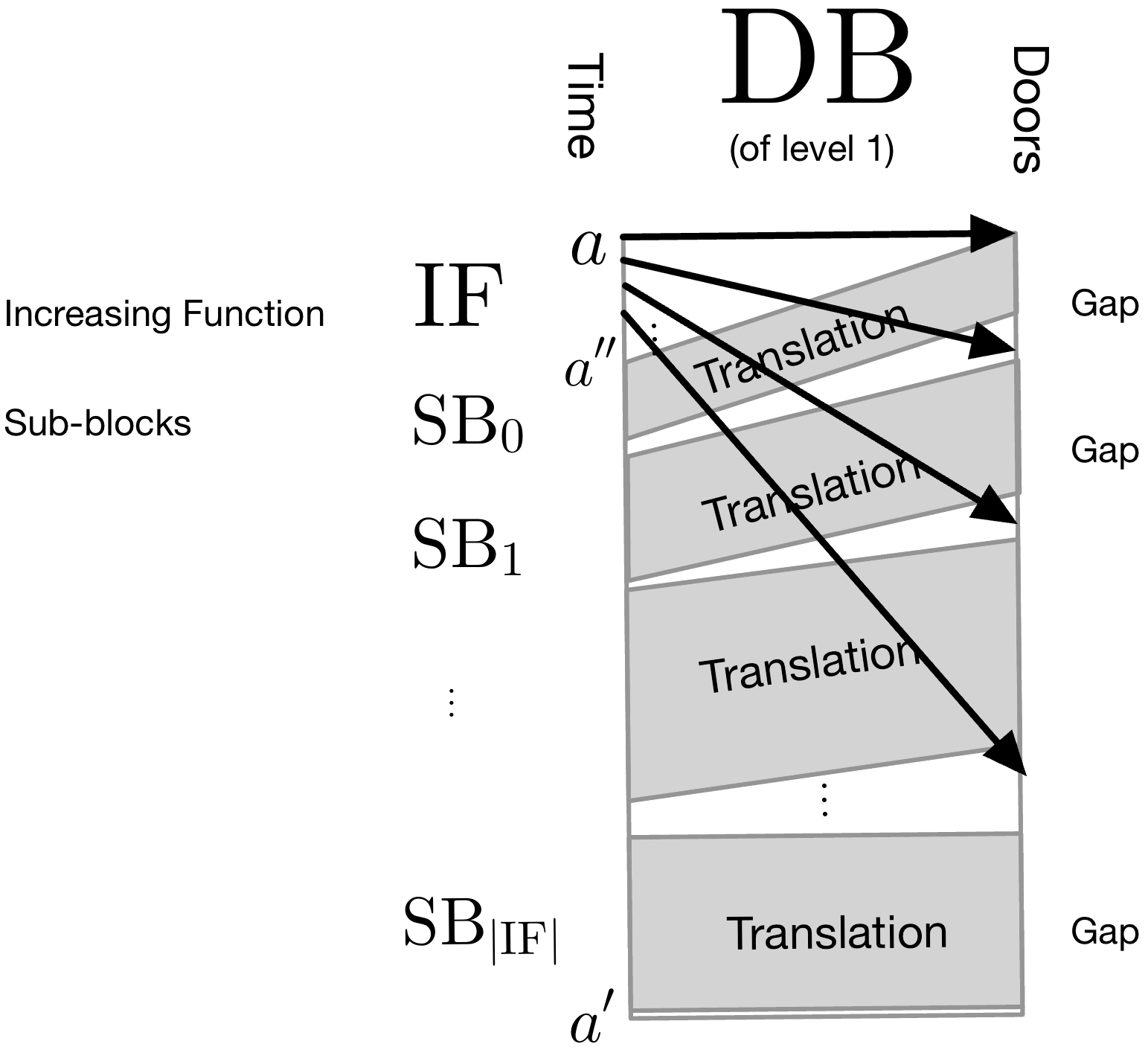}
    \caption{A $\DB$ of level 1 is a permutation that begins as an increasing function $\IF$, followed by a sequence of $|\IF|$ many sub-blocks $\SB$ that fill the gaps between $\IF$'s doors. At level 1, each $\SB$ is a translation function, and considered to be a degenerate disordered block of level 0, equal to its own increasing function and containing no sub-blocks.}
    \label{fig:DB}
\end{figure}

\subsubsection{Outcomes} \label{sec:outcome1}
The $s$-th disordered block $\DB$ shall be used to satisfy $\G_s$ and $P_{\f,s}$. We formalize in Algorithm $\oc$ the \emph{outcome of $\f$ within $\DB$}. Following earlier arguments, this outcome should be sparse if and only if $\f$ opens fewer than $s$ many doors at some $\SB\subblock\DB$. $\oc$ returns the first $\SB$ that $\f$ is sparse at. If none is found, the function will return $\NULL$ to indicate a dense outcome.

\begin{algorithm} \label{alg:outcome1}
    \caption{$\SB =\oc(\f,\DB)$\\ \Comment{Finds a sub-block that $\f$ is sparse at}}
    \begin{algorithmic}[1]
        \For{\textbf{each} $\SB\subblock\DB$}
            \If{$|\bb{Im}(\f)\cap\SB|<s$} \Comment{$\f$ opened $<s$ doors in $\SB$} \label{l:few-doors0}
                \State \Return $\SB$ \Comment{$\f$ is sparse within $\DB$ at $\SB$} \label{l:outcome-sparse0}
            \EndIf
        \EndFor
        \State \Return $\NULL$ \Comment{$\f$ is dense in $\DB$, or $\DB$ has no sub-blocks ($\DB$ is of level 0)} \label{l:outcome-dense0}
    \end{algorithmic}
\end{algorithm}

\begin{figure}[h]
    \centering
    \includegraphics[width=1.0\textwidth]{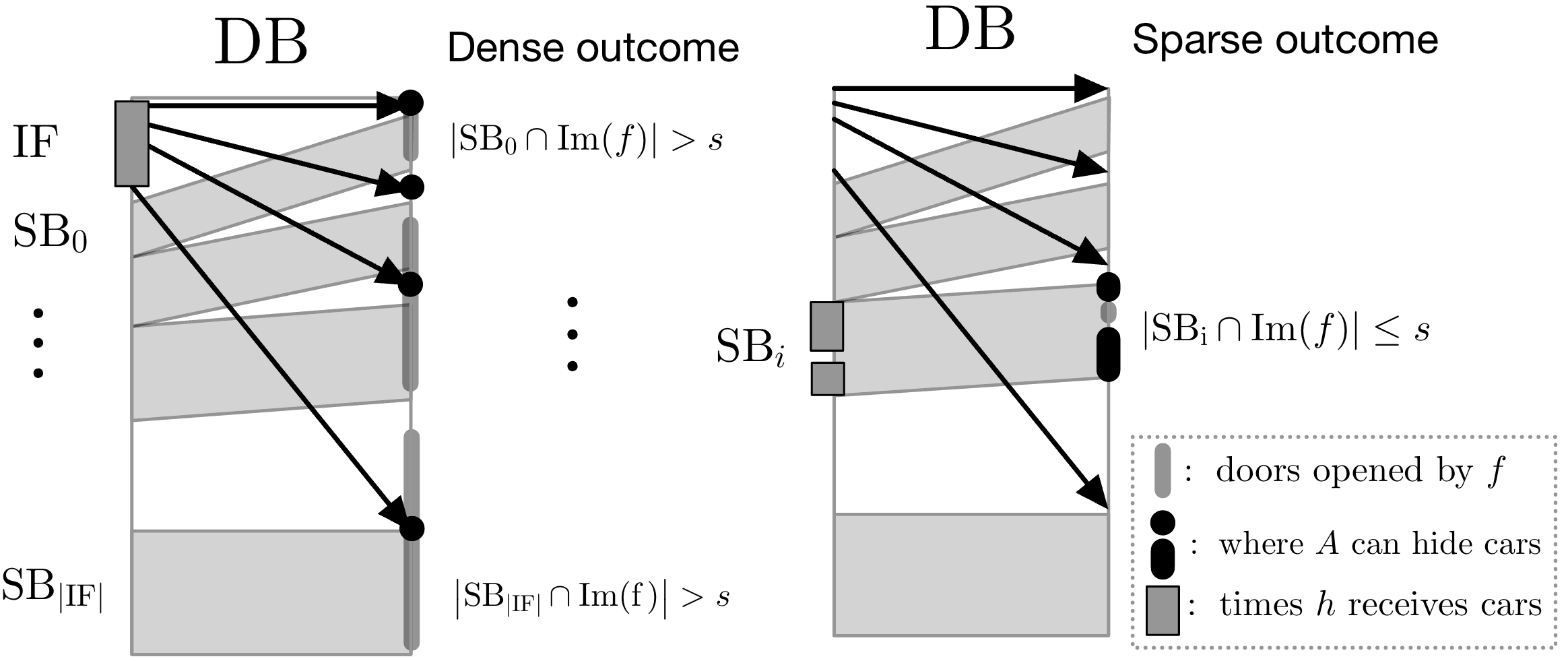}
    \caption{Working within $\DB$, opponent $\f$ has a dense outcome if $\f$ opens more than $s$ many doors at all $\SB\subblock\DB$, and has a sparse outcome otherwise. If dense, $\A$ will put its cars at $\IF\ll\DB$, and if sparse at $\SB$, $\A$ will is to put its cars at $\SB$ minus $\f$'s opened doors. In either cases, $\h$ will receive many cars where they are placed, while $\f$ rarely receives any.}
    \label{fig:outcomes-1}
\end{figure}

\subsubsection{Strategy} \label{sec:strategy1}
At stage $s$, we work only within the $s$-th disordered block $\DB$ and construct $\A\restriction\DB$. We also \emph{update} the following two sets by adding elements to them:
\begin{itemize}
    \item $\cars$ (\emph{Car restriction}) $:=$ doors that must contain cars,
    \item $\goats$ (\emph{Goat restriction}) $:=$ doors that must contain goats.
\end{itemize}

$\cars$ helps to ensure that $\h$ receives enough cars, while $\goats$ helps to ensure that $\f$ receives few cars. How should $\A\restriction\DB$ be constructed given the two possible outcomes of $\f$ within $\DB$? We formalize the strategy in this algorithm:
\[(\cars,\goats) =\cgone_s(\DB,\f,\cars,\goats),\]
which updates $\cars$ and $\goats$ to fill the doors of $\DB$. The algorithm is defined recursively for easier generalizations to future constructions. Our recursive depth for now will not exceed one because we are handling only a single opponent.\\

Consider the strategy if $\f$ is dense in $\DB$ ($\cgone$, line~\ref{l:dense0}). As outlined earlier, we will put cars within the increasing part of $\DB$ (line~\ref{l:dense-cars0}), making sure to fill the doors at all sub-blocks $\SB\subblock\DB$ with goats (line~\ref{l:dense-goats0}). On the other hand, if $\f$ is sparse in $\DB$ (line~\ref{l:sparse0}), say at $\SB$ (line~\ref{l:sparse-SB0}), we will ensure that no cars are placed outside of $\SB$ or within the doors in $\SB$ opened by $\f$ (line~\ref{l:sparse-goats0}). Some doors in $\SB$ remain unfilled, so we recurse into $\SB$ to fill them (line~\ref{l:recurse}). Since we only have one opponent, $\SB$ will not contain sub-blocks, so when we recurse into $\SB$, the outcome will be vacuously dense (line~\ref{l:dense0}). The algorithm will then put the cars into $\SB$'s increasing function (line~\ref{l:dense-cars0}), which is exactly $\SB$ itself, as desired.

\begin{algorithm} \label{alg:strategy0}
\caption{$(\cars,\goats) =\cgone_s(\DB,\f,\cars,\goats)$\\ \Comment{Fill $\DB$'s doors given opponent $\f$}}
\begin{algorithmic}[1]
    \small
    \State $\SB\gets\oc(\f,\DB)$ \label{l:sparse-SB0} \Comment{Find a sub-block that $\f$ is sparse at}
    \If{$\SB\neq\NULL$} \Comment{$\f$ is sparse at $\SB$} \label{l:sparse0}
        \State $\goats\gets \goats\cup(\DB-\SB)\cup(\Image(\f)\cap\SB)$ \label{l:sparse-goats0} \Comment{Put goats outside $\SB$ and in $\f$'s doors}
        \State \Return $\cgone_s(\SB,\f,\cars,\goats)$ \label{l:recurse0} \Comment{Recurse to finish filling $\SB$'s doors}
    \Else \Comment{All $\f\in\F$ are dense in $\DB$, or $\DB$ has no sub-blocks} \label{l:dense0}
        \State $\IF\gets$ the increasing part of $\DB$ \Comment{If $\DB$ is of level 0, $\IF=\DB$}
        \State $\cars\gets \cars\cup\IF-\goats$ \label{l:dense-cars0} \Comment{Put cars in increasing part}
        \State $\goats\gets \goats\cup\DB-\cars$ \label{l:dense-goats0} \Comment{Put goats in all other doors}
        \State \Return $\cars,\goats$ \Comment{Doors in $\DB$ are filled. Terminate recursion.} \label{l:terminate0}
    \EndIf
\end{algorithmic}
\end{algorithm}

How large should the increasing function and sub-blocks of $\DB$ be in order for $\h$ to receive enough cars? If $\f$ is sparse at $\SB\subblock\DB$, the cars will all be at $\SB$, minus the less than $s$ many doors opened by $\f$. Now $\h$ opens all the doors of $\SB$ in a single time-interval. To satisfy $\G_s$, we want $|\SB|$ to be large enough so that in spite of the disappointment from less than $s$ many goats, the cars received by $\h$ will exceed the number of doors they have opened so far by a factor of at least $s$. In fact, we shall account for $s^3$ many goats, not just $s$, to help us generalize to future constructions. Since we have defined our strategy recursively, the cars are placed after we have recursed into $\SB$ (line~\ref{l:recurse0}), at line~\ref{l:dense-cars0}, where $\SB$ is itself an increasing function. Thus, the following \emph{largeness condition} on $\SB$ can satisfy $G_s$ via $n=\min(\dom(\SB))+s^3$:
\begin{equation} \tag{largeness condition} \label{eq:large-condition}
    |\SB|\geq [\min(\dom(\SB))+s^3]\times s.
\end{equation}
Similarly, if $\f$ is dense in $\DB$, the cars will be placed in $\IF\ll\DB$, and the same \ref{eq:large-condition} will suffice for $G_s$.

\subsubsection{Construction} \label{sec:construct-db}
$h$ needs to be computable. Hence, the $\DB$'s cannot in general be constructed with respect to the opponents they will defeat. We shall construct $\DB_s$ at stage $s$ using an effective procedure, separate from the construction of $A$, which uses $\emptyset'$ as an oracle. We construct each $\DB$ by induction on nestedness levels $l$, as we have illustrated in Figures~\ref{fig:DB} and~\ref{fig:DB-r}. The pseudo-code for constructing $\DB_s$ is in this algorithm:
\[\DB = \con_s(l,t,d).\]
Again, the algorithm is recursive so that we can generalize easily to future constructions. With only one opponent, our $\DB$s' levels do not exceed 1, bounding the recursive depths by one. $\con_s(t,d,l)$ returns at stage $s$ a disordered block of level $l$, which is an injection from time-interval $[t,\_]$ to door indices $[d,\_]$ of the same length.

As described earlier, the $\DB$ ($\con$, line~\ref{l:DB}) begins with an increasing function $\IF\ll\DB$ (line~\ref{l:IF}) whose size satisfies the \ref{eq:large-condition}. $\IF$ is followed by a sequence of $|\IF|$ many sub-blocks $\SB$ (lines~\ref{l:SB}). $\IF$ maps to the door above each sub-block (line~\ref{l:IF-door}). Each $\SB$ is considered a disordered block of level $l-1$, so we recurse into $\SB$ for its construction (line~\ref{l:SB-recurse}). When we reach the terminating depth at $l=0$, the disordered block will not have sub-blocks (line~\ref{l:TB}), and is considered its own increasing function. In other words, the disordered block is simply a translation function.

\begin{algorithm} \label{alg:db}
    \caption{$\con_s(l,t,d)$\\ \Comment{Constructs level $l$ $\DB$ with starting time $t$ and first door $d$}}
    \begin{algorithmic}
        \State \textbf{Define} $\IF: t+[0,(t+s^3)s]\to\omega$ \Comment{$|\IF|$ satisfies \ref{eq:large-condition}} \label{l:IF}
        \State $t'\gets \max(\dom(\IF))+1$ \Comment{First time-stamp of next sub-block}
        \State $d'\gets d+1$ \Comment{First door of next sub-block}
        \For{$i=0$ \textbf{to} $|\IF|$} \Comment{Construct $\SB_i$} \label{l:SB}
            \State $\IF(t+i) =d'-1$ \Comment{Map to the door above $\SB_i$} \label{l:IF-door}
            \If{$l\geq 1$} \Comment{$\SB_i$ is a disordered block of level $l-1$}
                \State $\SB_i =\con_s(t',d',l-1)$ \Comment{Recurse to construct $\SB_i$} \label{l:SB-recurse}
                \State $t'\gets \max(\dom(\SB_i))+1$ \Comment{Update first time-stamp}
                \State $d'\gets \max(\Image(\SB_i))+2$ \Comment{Update first door}
            \Else \Comment{$\DB$ is of level 0}
                \State $\SB_i=\emptyset\to\emptyset$ \Comment{A level zero $\DB$ has no sub-blocks} \label{l:TB}
                \State $d'\gets d'+1$
            \EndIf
        \EndFor
        \State $\DB =\IF^\frown \SB_0\;^\frown \SB_1\;^\frown \ldots ^\frown \SB_{|\IF|}$ \Comment{Equation~\eqref{eq:concat-db}} \label{l:DB}
        \State \Return $\DB$
    \end{algorithmic}
\end{algorithm}

We are now ready to construct $\h$. Our current set of opponents are
\[\F=\{\f_0\}.\]

\underline{Constructing $\h$:} Initialize $\h:\emptyset\to\emptyset$. At stage $s$,
\begin{enumerate}
    \item $\DB_s=\con_s(|\F|,\max(\dom(\h)),\max(\Image(\h)))$.
    \item Concatenate $\DB_s$ to $\h$.
    \item Go to stage $s+1$.
\end{enumerate}

Using $\emptyset'$ as an oracle, we can now construct $\A$.\\

\underline{Constructing $\A$:} Initialize $\A=\cars=\goats=\emptyset$. At stage $s$:
\begin{enumerate}
    \item Let $\DB_s$ be the $s$-th disordered block of $h$.
    \item Use $\emptyset'$ to get $\sigma=\f_0\restriction\DB_s$.
    \item $\cars,\goats =\cgone_s(\DB_s,\sigma,\cars,\goats)$.
    \item $A\restriction\DB_s =\cars\restriction\DB_s$.
    \item Go to stage $s+1$.
\end{enumerate}

\subsubsection{Verification} \label{sec:verify1}
We want to show for all $s\geq 1$ that $\G_s$ and $\Pos_s=\Pos_{\f_0,s}$ are satisfied. Note that $\G_0$ and $\Pos_0$ will hold from $\G_1$ and $\Pos_1$ respectively. When considering $\Pos_s$, we can assume $\f=\f_0$ is total, otherwise the requirement is satisfied vacuously. Also, it will be enough to show that the inner-clause of $\Pos_s$ holds when we work within $\DB_s$: This implies that for all $t\geq s$, the inner-clause of $\Pos_t$, and hence of $\Pos_s$, is satisfied. Putting the effects from $\DB_{\geq s}$ together, we satisfy $\Pos_s$ via
\[m=\min(\dom(\DB_s)).\]

We work within $\DB_s$ and show how $A\restriction \DB_s$ satisfies $\G_s$ and also the inner-clause of $\Pos_s$. If $\f$ was dense in $\DB_s$, we would have put cars at $\IF\ll\DB_s$ ($\cgone$, line~\ref{l:dense-cars0}). The \ref{eq:large-condition} of $\IF$ ($\con$, line~\ref{l:IF}) ensures that any lack of cars received by $\h$ so far would be compensated at least $s$ many times, thereby satisfying $\G_s$. The denseness of $\f$'s outcome means for every car $\f$ might receive from opening a door in $\IF$, $\f$ would receive more than $s$ many goats in the sub-block $\SB\subblock\DB_s$ that follows, thereby satisfying $\Pos_s$.\\

On the other hand, if $\f$ was sparse in $\DB_s$, say at $\SB\subblock\DB_s$, then cars would be placed in $\SB$ minus the doors opened by $\f$ in $\SB$. The subtraction satisfies $\Pos_s$ immediately. The \ref{eq:large-condition} of $\SB$ would again give $\h$ at least $s$ times as many cars as they have received so far, in spite of the subtraction of no more than $s$ cars, thereby satisfying $\G_s$.

\subsection{Beating two $f$'s using two-nested $\DB$'s}
What happens if we play against two opponents $\f_0$ and $\f_1$? Working within some $\DB$ constructed at stage $s$, $\f_0$ and $\f_1$ can have sparse or dense outcomes, giving $2^2$ possibilities.\\

The easiest situation is when both outcomes are dense. Then by the original strategy against one opponent, $\A$ will put cars at $\IF\ll\DB$, and the arguments for this strategy apply to both opponents.\\

So consider the situation when one opponent, say $\f_0$, has a sparse outcome, say at $\SB\subblock\DB$. If $\f_0$ was the only opponent, we would have put all the cars in $\SB$ minus the doors opened by $\f_0$. If $\f_1$ is also sparse at $\SB$, this original strategy remains valid as long as we subtract also the less than $s$ many doors opened by $\f_1$, and ensure that the \ref{eq:large-condition} accounts for this bounded subtraction.\\

But what if $\f_1$ is dense at $\SB$? $\G_s$ wants us to put many cars into $\SB$. Yet, $\Pos_{\f_1,s}$ wants us to put a car only after $\f_1$ has received at least $s$ many consecutive goats. The solution is to make $|\SB|$ so large as to still have enough cars for $\G_s$ even if we skip over roughly $s$ many doors before adding a car in $\SB$. To ensure that $\h$ opens mostly these car-filled doors consecutively, we shall make $\SB$ begin as an increasing function that maps to those doors, thereby satisfying $\G_s$. The function will be followed by a sequence of bijections that fill the gaps of the increasing part. Such a structure is like a level 1 disordered block, so we give $\SB$ the same structure as the $\DB$ constructed against one opponent.

\subsubsection{Two-nested $\DB$ structure} \label{sec:construct-db2}
Summarizing, to beat two opponents, we use 2-nested disordered blocks, where every sub-block is no longer a translation function, but a 1-nested disordered block. Letting $\DB^k$ denote a $k$-nested $\DB$, $h$ will be a series of $\DB^2$'s:

\begin{align} \label{eq:concat-h2}
    \h =\DB^2_0 \;^\frown \DB^2_1 \;^\frown \DB^2_2 \;^\frown \ldots,
\end{align}
where each $\DB^2$ is a bijection between intervals (Figure~\ref{fig:DB-r} with $r=1$), and has the structure (abusing concatenation notation)
\begin{align} \label{eq:concat-db2}
    \DB^2 =\IF^{\frown} \DB^1_0 \;^\frown \DB^1_1 \;^\frown \ldots \DB^1_{|\IF|},
\end{align}

where
\[\IF = \text{ the increasing function that points to the door before each } \DB^1_i,\]
and has a domain interval long enough to satisfy the \ref{eq:large-condition}, and each $\DB^1$ is a 1-nested disordered block from Section~\ref{sec:DB} whose components all satisfy the \ref{eq:large-condition}.\\

Formally, to construct $\DB^2$ with the structure of Equation~\eqref{eq:concat-db2} at stage $s$, we call $\con_s(2,t,d)$, whose starting time and door are $t$ and $d$ respectively. The overall construction of $\h$ that satisfies Equation~\eqref{eq:concat-h2} was given at the end of Section~\ref{sec:construct-db}, but with the following set of opponents
\[\F=\{\f_0,\f_1\}.\]

To construct $\A$, we modify $\cgone$ to $\cgr$, which can take in multiple opponents $\F$. Following earlier description, in an easy case, all opponents are dense in $\DB$ ($\cgr$, line~\ref{l:dense1}). Our strategy is to put cars only in $\IF\ll\DB$ (lines~\ref{l:dense-cars1}, \ref{l:dense-goats1}).\\

But if one of them, say $\f$, is sparse in $\DB$ (line~\ref{l:sparse1}), say at $\SB\subblock\DB$ (line~\ref{l:sparse-SB1}), we will keep cars outside of the doors opened by $\f$ and the doors not in $\SB$ (line~\ref{l:sparse-goats1}). It remains to fill the doors in $\SB$ different from those opened by $\f$ (line~\ref{l:recurse1}). $\f$ never opens these remaining doors, so we can remove $\f$ from future consideration within this $\DB$. To fill these doors, we recurse into $\SB$ (line~\ref{l:recurse1}), which will consider the outcome of $\g\in\F-\{f\}$ in $\SB$. Since we have not forced any doors in $\DB$ to contain cars so far, $\g$ would not have received any cars from $\DB$ yet, allowing us to add cars to $\SB$ without worrying about the combined effects with doors outside $\SB$.\\

In this recursive depth, the algorithm will treat $\SB$ like a disordered block of level one, behaving just like in $\cgone$ with opponent $\g$: If $\g$ is dense in $\SB$, we would put cars only at $\SB$'s increasing part (line~\ref{l:dense-cars0}) and terminate. But if $\g$ is sparse in $\SB$ at $\SB'\subblock\SB$ (line~\ref{l:sparse-SB1}), we keep cars out of $\g$'s doors and doors not in $\SB'$ (line~\ref{l:sparse-goats1}), then remove $\g$ from consideration and recurse to fill the doors of $\SB'$ (line~\ref{l:recurse1}), which is of level 0. This second depth recursion will have a vacuously dense outcome from $\SB'$'s lack of sub-blocks (line~\ref{l:dense1}), and therefore place the cars at the increasing part of $\SB'$, which is all of $\SB'$ itself (line~\ref{l:dense-cars1}).\\

\begin{algorithm} \label{alg:strategy1}
\caption{$(\cars,\goats) =\cgr_s(\DB,\F,\cars,\goats)$\\ \Comment{Fill $\DB$'s doors given opponents $\F$}}
\begin{algorithmic}[1]
    \small
    \State $\f\gets \mu(\f\in\F)[\oc(\f,\DB)\neq\NULL]$ \Comment{Find opponent with sparse outcome}
    \If{$\f$ exists} \Comment{Some $\f$ is sparse in $\DB$} \label{l:sparse1}
        \State $\SB\gets\oc(\f,\DB)$ \Comment{Find $\SB$ that $\f$ is sparse at} \label{l:sparse-SB1}
        \State $\goats\gets \goats\cup(\DB-\SB)\cup(\Image(\f)\cap\SB)$ \label{l:sparse-goats1} \Comment{Put goats outside $\SB$ and in $\f$'s doors}
        \State \Return $\cgr_s(\SB,\F-\{\f\},\cars,\goats)$ \label{l:recurse1} \Comment{Recurse to finish filling $\SB$'s doors}
    \Else \Comment{All $\f\in\F$ are dense in $\DB$, or $\DB$ has no sub-blocks, or $\F=\emptyset$} \label{l:dense1}
        \State $\IF\gets$ the increasing part of $\DB$ \Comment{If $\DB$ is of level 0, $\IF=\DB$}
        \State $\cars\gets \cars\cup\IF-\goats$ \label{l:dense-cars1} \Comment{Put cars in increasing part}
        \State $\goats\gets \goats\cup\DB-\cars$ \label{l:dense-goats1} \Comment{Put goats in all other doors}
        \State \Return $\cars,\goats$ \Comment{Doors in $\DB$ are filled. Terminate recursion.} \label{l:terminate1}
    \EndIf
\end{algorithmic}
\end{algorithm}

\begin{figure}[h]
    \centering
    \includegraphics[width=0.5\textwidth]{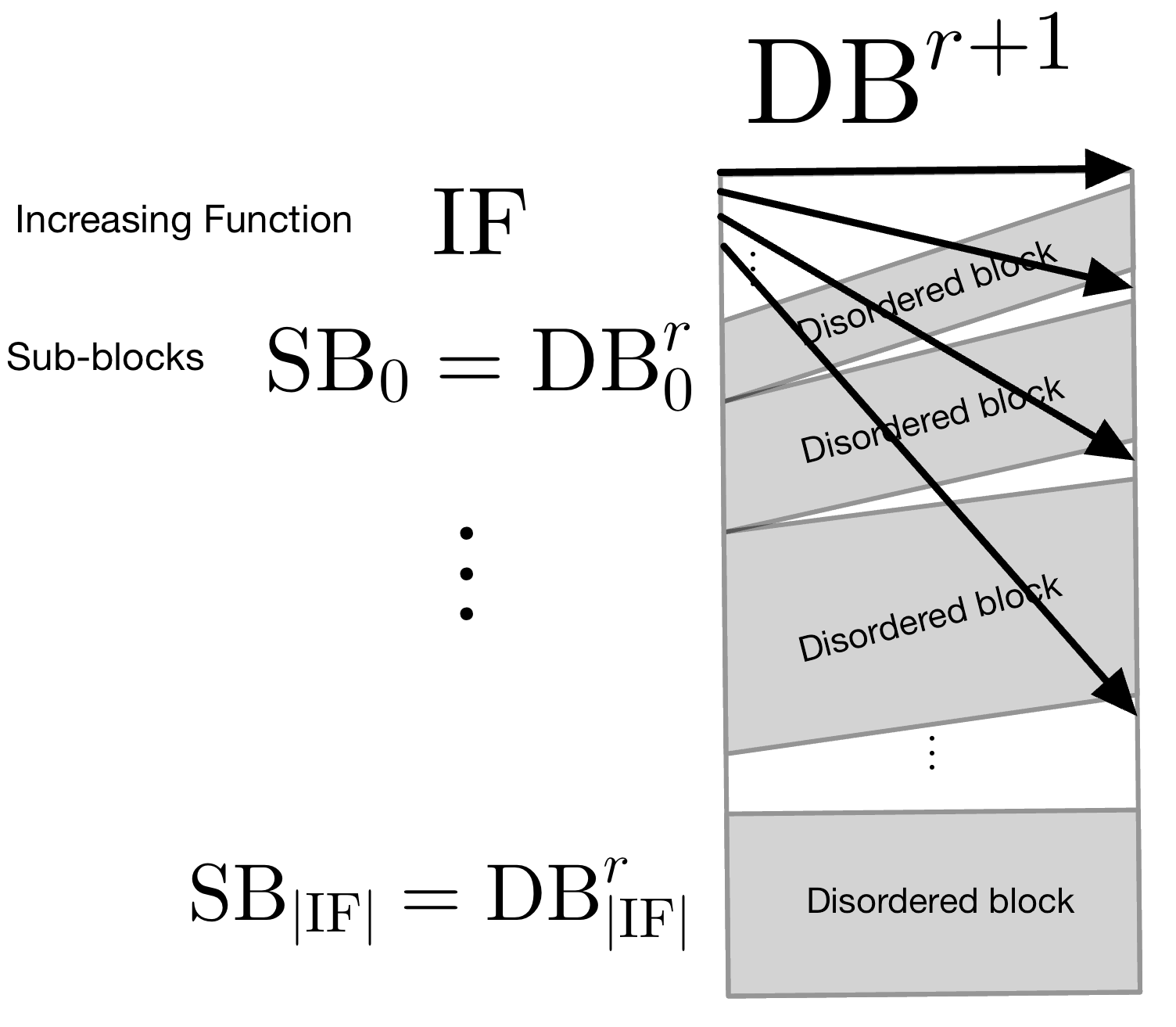}
    \caption{Given $r\geq1$, $\DB^{r+1}$ is a disordered block of level $r+1$. It is a bijection from time to door intervals that starts as an increasing function $\IF$, followed by $|\IF|$ many lower level disordered blocks $\DB^r$ that fill the gaps between $\IF$'s doors. The base structure $\DB^1$ is sketched in Figure~\ref{fig:DB}.}
    \label{fig:DB-r}
\end{figure}

\subsection{Beating $r$ many $f$'s by using $r$-nested $\DB$'s}
The ideas for beating two opponents can be generalized naturally to beating $r$ many. We shall show by induction on opponent number $r$ that a concatenation of $r$-nested $\DB$'s
\begin{align} \label{eq:concat-hr}
    \h =\DB^r_0 \;^\frown \DB^r_1 \;^\frown \DB^r_2 \;^\frown \ldots
\end{align}
can beat $\F=\{\f_0,\ldots,\f_{r-1}\}$.

Consider the addition of a new opponent $\f_r$ into $\F$. Following the strategy for the case of two opponents, if all opponents are dense in $\DB$ ($\cgr$, line~\ref{l:dense1}), we put cars only in $\IF\ll\DB$ (lines~\ref{l:dense-cars1}, \ref{l:dense-goats1}). But if some $\f$ is sparse in $\DB$ (line~\ref{l:sparse1}), say at $\SB\subblock\DB$ (line~\ref{l:sparse-SB1}), we fill doors opened by $\f$ and outside of $\SB$ with goats (line~\ref{l:sparse-goats1}), then recurse into $\SB$ to fill its empty doors, removing $\f$ from consideration (line~\ref{l:recurse1}).\\

By induction on $r$, to beat the remaining $r$ many opponents when we work within $\SB$, $\SB$ needs only to have the structure of a disordered-block of level $r$. In other words, $\DB$ needs to be a level $(r+1)$ disordered block $\DB^{r+1}$, which completes the sketch of the inductive step.

\subsubsection{$r$-nested $\DB$ structure} \label{sec:construct-dbr}
Assume we have the structure of $\DB^{r-1}$. To get the structure of $\DB^r$ (Figure~\ref{fig:DB-r}), we generalize the formulas of Section~\ref{sec:construct-db2}:

\begin{align} \label{eq:concat-dbr}
    \DB^r =\IF^{\frown} \DB^{r-1}_{0} \;^\frown \DB^{r-1}_{1} \;^\frown \ldots \DB^{r-1}_{|\IF|},
\end{align}

where
\[\IF = \text{ the increasing function that points to the door before each } \DB^{r-1}_i,\]
and has a domain interval long enough to satisfy the \ref{eq:large-condition}, and each $\DB^{r-1}$ is a level $r-1$ disordered block. 

To see how the \ref{eq:large-condition} makes each component of $\DB^r$ large enough for $\G_s$, we keep track of the number doors that are forced to contain goats in the component that $\h$ will get its cars from. We recurse into a sub-block component ($\cgr$, line~\ref{l:recurse1}) to narrow down the placement of cars (line~\ref{l:dense-cars1}) only when some $\f\in\F$ has a sparse outcome. Every time we recurse, the sub-block we zoom into must fill the doors opened by $\f$ with goats (line~\ref{l:sparse-goats1}). There are no more than $s$ many such doors due to sparseness, and the maximum depth of recursion cannot exceed $|\F|=r$. Hence, the final component where cars reside needs to account for no more than $rs$ many goats. For large enough $s$, $rs\leq s^2$, which will not exceed the $s^3$ many goats that the \ref{eq:large-condition} accounts for.\\

Formally, to construct $\DB^r$ with the structure of Equation~\eqref{eq:concat-dbr} at stage $s$, we call $\con_s(r,t,d)$, whose starting time and door are $t$ and $d$ respectively. The construction of $\h$ that satisfies Equation~\eqref{eq:concat-hr}, and of $\A$ that satisfies $\G$ and $\Pos$, was given at the end of Section~\ref{sec:construct-db}, but using $\cgr$ instead of $\cgone$, and using opponents $\F=\{\f_0,\ldots,\f_{r-1}\}$.

\subsubsection{Verification} \label{sec:verifyr}
We follow the outline of Section~\ref{sec:verify}. Let $\F=\{\f_0,\ldots,\f_{r-1}\}$. Given stage $s>r$, we work within $\DB_s=\DB^r_s$ and show how $A\restriction \DB_s$ satisfies $\G_s$ and the inner-clause of $\Pos_{\f,s}$ for every $\f\in\F$.\\

Note that when we call $\cgr_s(\DB_s,\F,\cars,\goats)$ at stage $s$ to fill the doors of $\DB_s$, the maximum recursion depth cannot exceed $|\F|=r$, because we remove one opponent every time we recurse ($\cgr$, line~\ref{l:recurse1}), and the algorithm must terminate (line~\ref{l:terminate1}) when there are no more opponents (line~\ref{l:dense1}). Let $\DB$ be the component of $\DB_s$ that we have zoomed into at the final depth, and let $\F_0\subseteq\F$ be the set of opponents that remain prior to terminating.\\

Let $\f\not\in\F_0$ and consider the cars received by $\f$ within $\DB_s$. $\f$ must have been removed at a shallower depth (line~\ref{l:recurse1}), which means $\f$ was sparse at some component of $\DB_s$ that contains $\DB$ (line~\ref{l:sparse-SB1}). Within $\DB$, $\f$ cannot receive cars, because just prior to removing $\f$, we have forced the the doors that $\f$ will visit in the component to contain goats (line~\ref{l:sparse-goats1}). These goat restrictions are respected at the terminating depth when we select doors to contain cars (line~\ref{l:dense-cars1}). Also, no car resides outside $\DB$ because we only put cars in the terminating component $\DB$ (line~\ref{l:dense-cars1}). So $\f$ receives no cars in $\DB_s$, which satisfies $\Pos_{\f,s}$ trivially.\\

Now let $\f\in\F_0$ and consider the cars that $\f$ will receive. Again, outside $\DB$, no cars can be obtained. At the terminating depth (line~\ref{l:dense1}) when we place cars in $\DB$, if $\F_0$ is non-empty, then $\DB$ must contain sub-blocks, because $\DB_s$ is of level $r+1$ and the depth of recursion does not exceed $r$. Thus if $\f$ exists, it must be dense in $\DB$, and we would place the cars in $\IF\ll\DB$ (line~\ref{l:dense-cars1}). This placement ensures that every possible car $\f$ receives from $\IF$ would be followed by at least $s$ many goats in the next sub-block $\SB\subblock\DB$, thereby satisfying $\Pos_{\f,s}$.\\

Finally, consider the cars that $\h$ will receive. These cars reside in $\DB$, whose components satisfy the largeness condition. The condition tolerates $s^3$ many goats within $\DB$, which is more than the up to $sr$ many goats enforced by $\F$.

\subsection{Beating all $f$'s by using increasingly nested $\DB$'s} \label{sec:DB-all}
With countably many opponents, our strategy is to handle the first $s$ many of them
\[\F_s:=\{\f_0,\f_1,\ldots,\f_{s-1}\}\]
at stage $s$. Since $\F_s$ can be beaten using disordered blocks $\DB^l$ of levels $l\geq s$, we construct $\h$ to be a series of increasingly nested disordered blocks:
\begin{align} \label{eq:concat-h-all}
    \h=\DB^0\;^\frown \DB^1\;^\frown \DB^2\;^\frown \ldots,
\end{align}
where the structure of $\DB^l$ was given in Section~\ref{sec:construct-dbr}. The construction of $\h$ and $\A$ are is also similar to the ones at the end of Section~\ref{sec:construct-db}.

\underline{Constructing $\h$:} Initialize $\h:\emptyset\to\emptyset$. At stage $s$,
\begin{enumerate}
    \item $\DB_s =\con_s(s,\max(\dom(\h)),\max(\Image(\h)))$.
    \item Concatenate $\DB_s$ to $\h$.
    \item Go to stage $s+1$.
\end{enumerate}

\underline{Constructing $\A$:} Initialize $\A=\cars=\goats=\emptyset$. At stage $s$:
\begin{enumerate}
    \item Let $\DB_s$ be the $s$-th disordered block of $h$.
    \item Use $\emptyset'$ to get $\F_s:=\{\f_r\restriction\DB_s: r\leq s\}$.
    \item $\cars,\goats =\cgr_s(\DB_s,\F_s,\cars,\goats)$.
    \item $A\restriction\DB_s =\cars\restriction\DB_s$.
    \item Go to stage $s+1$.
\end{enumerate}

\subsubsection{Verification} \label{sec:verify}
To satisfy $\Pos_{\f_r,t}$, work within $\DB^s$ for any $s>\max(r,t)$; then apply the same argument for the $\Pos$-requirements as in Section~\ref{sec:verifyr}. To satisfy $\G_s$, work within $\DB^s$; $\h$ will receive enough cars from $\DB^s$ due to the largeness condition of $\DB^s$'s components, just like in Section~\ref{sec:verifyr}.

\subsection{Not Church Stochastic}
Note that this construction produces a set which is not Church stochastic: since $A$ is computably small, it has density $0$. So for it to be Church stochastic, it must be stochastic with parameter $0$. Then consider the selection rule which checks every door in order, and upon finding a car in the range of any $IF$ or level 0 $SB$ under $h$, selects the remainder of the block. This is an orderly, adaptive strategy which selects a sequence of positive density on $A$: by construction, any such block containing some element of $A$ contains enough elements to have high density after inverting $h$. 

\subsection{Extending to positive density} \label{sec:alpha}

Proving separations for $0$ makes sense from the standpoint of computability theory, as each of these is a notion of immunity. However, from a standpoint of the original motivating game, $0$ is uninteresting as a probability. Thus we shall show that the separation between computable density and intrinsic density holds for every $\alpha\in (0,1]$ by way of the separation for $0$.

\begin{lemma} \label{lemma:alpha}
If there is $X$ such that $X$ is computably small but not intrinsically small, then for all $\alpha$ there is some $Z$ with $R(Z)=\alpha$ but $Z$ does not have intrinsic density $\alpha$.
\end{lemma}
\begin{proof}
Let $R(X)=0$ and $R(Y)=\alpha$. Then it follows immediately that $R(Y\setminus X)=R(Y\cup X)=\alpha$. Now suppose $X$ is not intrinsically small with witnessing permutation $\pi$. If $Y\setminus X$ does not have intrinsic density $\alpha$, then we are done. Thus assume it does have intrinsic density $\alpha$, and consider $\rho_n(\pi(Y\cup X))$. We have that
\[\rho_n(\pi(Y\cup X))=\rho_n(\pi(Y\setminus X))+\rho_n(\pi(X))\]
Then there exists $k$ such that $\alpha-\frac{q}{2}<\rho_n(\pi(Y\setminus X))<\alpha+\frac{q}{2}$ for all $n>k$ because $P(Y\setminus X)=\alpha$. By assumption, there is $q>0$ such that $\rho_n(\pi(X))>q$ infinitely often. But then, for any $m>k$ such that $\rho_m(X)>q$, we have that 
\[\rho_m(Y\cup X)=\rho_m(Y\setminus X)+\rho_m(X)>\alpha-\frac{q}{2}+q=\alpha+\frac{q}{2}\]
Thus $\overline{\rho}(\pi(Y\cup X))\geq \alpha+\frac{q}{2}$, so in particular $\rho(\pi(Y\cup X))\neq\alpha$. Thus $Y\cup X$ does not have intrinsic density $\alpha$.
\end{proof}

Applying Lemma~\ref{lemma:alpha} by letting $X$ be as in Theorem~\ref{thm:disorder-stronger-order}, we get:
\begin{lemma}
    Given $\alpha\in[0,1]$, there exists a set of computable density $\alpha$ which is not intrinsically dense.
\end{lemma}
\section{Weak stochasticity} \label{sec:disorder-not-stronger-than-weak-adaptable}

Having defined weak stochasticity above, we now prove that it is separate from Church (and therefore MWC) stochasticity, and that it does not imply intrinsic density.

\begin{theorem}
    \label{thm:weak-non-church}
    There is a set $A$ which is 0-weakly stochastic but not 0-Church stochastic.
\end{theorem}

\begin{proof}
    Let $A$ be 0-weakly stochastic. Then we claim $A\oplus A$ is also 0-weakly stochastic, which we shall prove by contrapositive. Suppose that $A\oplus A$ is not 0-weakly stochastic. Then we shall show that $A$ is not either. Let $f$ be the skip rule for which $\overline{\rho}(f(A\oplus A))>0$. Then define the skip rule $g$ to select $k$ on the skip sequence $\sigma$ whenever $f$ selects $2k$ or $2k+1$ on $\gamma$ (doing nothing if $f$ attempts to select $2k+1$ following $2k$), where $\gamma$ is the skip sequence generated by $f$ on $\sigma\oplus\sigma$.
    \\
    \\
    Then $\overline{\rho}(g(A))\geq \frac{\overline{\rho}(f(A\oplus A)}{2}>0$: for each bit selected by $g(A)$, there are either one or two copies of the same bit in $f(A\oplus A)$. Thus $A$ is not $0$-weakly stochastic as desired.
    \\
    \\
    Finally, it is straightforward to prove that $A\oplus A$ is not Church stochastic for any infinite, co-infinite $A$. For example, see \cite{thesis} Proposition 3.23.
\end{proof}

Note that this proof only works for weak stochasticity $0$: if $A$ is $r$-weakly stochastic for positive $r$, then the skip rule which selects every even bit at index $2n$ and selects the odd bit at index $2n+1$ only if the former is a $1$, selects a sequence of density $\frac{2r}{r+1}$. In particular, for $0<r<1$, this gives us an example of a set which has intrinsic density $r$ but is not $r$-weakly stochastic. The following proof extends this to include $r=0$:
\\
\par
\begin{theorem}
    There is an intrinsically small set $X$ which is not $0$-weakly stochastic. 
\end{theorem}

\begin{proof}
    We shall prove that the intrinsic density of $X$ is $0$ using the equivalent definition given by Astor \cite{intrinsicdensity}: $X$ is intrinsically small if and only if $\rho(\pi(X))=0$ for every computable permutation $\pi$.
    \\
    \par
    Given $n$, let $b_n=\frac{4^n-1}{3}$, which is the sum of the first $n$ powers of $4$. Define the $n$-th ordered block to be $[b_n,b_{n+1})$. The $i$-th sub-block of the $n$-th (ordered) block is $[b_n+i2^n,b_n+(i+1)2^n)$ for $0\leq i<2^n$. In other words, the $n$-th block has length $4^n$, and each of its $2^n$ sub-blocks has length $2^n$.

    \begin{lemma}
    \label{lemma:infblock}
        Let $A$ be any set which, for infinitely many $n$, contains exactly one sub-block of the $n$-th ordered block. Then $A$ is not $0$-weakly stochastic.
    \end{lemma}    

    \begin{proof}
        Consider the skip rule $f$ which goes block by block and selects bits of $A$ in the $n$-th block by checking the first bit of each sub-block. If it finds a $1$ in the $i$-th block, then it selects the entire $i$-th sub-block, then proceeds to the next block. If it finds a $0$ in every sub-block, then it proceeds to the next block.
        \\
        \\
        Note that in the $n$-th block, this skip rule selects at most $2^n$ 0's and selects $2^n$ 1's whenever it finds a 1. So if $A$ contains a sub-block at level $n$, then the number of selected $0$'s is at most $2(2^{n}-1)$: 
        \[\mathop{\sum}_{i=0}^n 2^i=2^n-1\]
        as the sum of the maximum possible number of 0s selected from the previous blocks, and $2^n-1$ 0's as the maximum number of first bits in each sub-block of the $n$-th block. It selects at least $2^n$ 1's (possibly more from previous sub-blocks), so if $s=3\cdot 2^n-2$, we have
        \[\rho_s(f(A))=\frac{|f(A)\upharpoonright s|}{s}\geq\frac{2^n}{s}\geq \frac{2^n}{3\cdot 2^n}=\frac{1}{3}\]
        Thus, whenever $A$ contains exactly one sub-block of the $n$-th ordered block, the density of $f(A)$ is at least $\frac{1}{3}$. Therefore, $\overline{\rho}(f(A))\geq\frac{1}{3}>0$, so $A$ is not $0$-weakly stochastic.
    \end{proof}

    Therefore, it suffices to prove that there is some set $X$ satisfying the requirements of this lemma which is also intrinsically small. We shall construct $X$ using $\emptyset''$ as an oracle to enumerate all of the computable permutations $\pi_i$. (The set of $e$ such that $\varphi_e$ is a permutation is a $\Pi_2^0$ index set.) The following technical lemma will allow us to ensure that the density under a given permutation is small.

    \begin{lemma}
    \label{lemma:finblock}
        Let $\pi$ be a permutation of $[0,4^n)$. Define the $i$-th sub-block to be $[i2^n,(i+1)2^n)$. Then the number of sub-blocks $X$ for which there is an $s\in [1,4^n]$ with 
        \[\rho_s(\pi(X))>\frac{1}{n}\]
        is at most $O(n^2)$.
    \end{lemma}

    \begin{proof}
        We say a sub-block is \emph{big} if there is such an $s$. Each big sub-block has a minimal witness $s$, i.e. 
        \[\rho_s(\pi(X))=\frac{|\pi(X)\upharpoonright s|}{s}>\frac{1}{n}\]
        but
        \[\rho_t(\pi(X))=\frac{|\pi(X)\upharpoonright t|}{t}\leq\frac{1}{n}\]
        for all $t<s$. If $|\pi(X)\upharpoonright s|=k$, we say $X$ is $k$-big. 
        \\
        \\
        We shall count the maximum number of $k$-big sub-blocks we can have. If
        \[\rho_s(\pi(X))=\frac{|\pi(X)\upharpoonright s|}{s}>\frac{1}{n}\]
        then we have
        \[n|\pi(X)\upharpoonright s|=nk>s\]
        In other words, a $k$-big sub-block must have minimal $s<nk$. So there can be at most $n$ such sub-blocks, as the sub-block images are disjoint and we need $k$ elements in each one.
        \\
        \\
        However, the maximum number of $k$-big sub-blocks cannot be achieved for each $k$ simultaneously: If the maximum number of 1-big sub-blocks is achieved, then each element of $[0,n)$ is in a different sub-block's image. However, each 2-big sub-block needs to contain two elements of $[0,2n)$,  but the only remaining elements are $[n,2n)$ as $[0,n)$ has been exhausted by the 1-big sub-blocks. Therefore, there can be at most $\frac{n}{2}$ 2-big sub-blocks when there are $n$ 1-big sub-blocks.
        \\
        \\
        In general, the maximum number of big sub-blocks occurs when we maximize the number of $k$-big sub-blocks given that the number has been maximized for all smaller $k$: the maximum number of big sub-blocks occurs when as few elements as possible are used to make them big. Thus, the number of $k$-big sub-blocks in this context is at most $\frac{n}{k}$, as all elements of $[0,(k-1)n)$ have been exhausted by the sub-blocks which are big for numbers smaller than $k$, and thus each $k$-big sub-block needs $k$ elements from $[(k-1)n,kn)$.
        \\
        \\
        As each sub-block contains $2^n$ elements, the largest $k$ for which $X$ can be $k$-big is $2^n$. Thus the number of big sub-blocks is at most
        \[\mathop{\sum}_{k=1}^{2^n}\frac{n}{k}=n\mathop{\sum}_{k=1}^{2^n}\frac{1}{k}\]
        By the Euler-Maclaurin formula, the harmonic numbers are $O(\log n)$, so the subsequence indexed by powers of 2 is $O(n)$. Thus the above is $O(n^2)$.
    \end{proof}

    In general, we cannot rely on arbitrary computable permutations to nicely map finite blocks to finite blocks. However, observe the following: Consider $\pi$ a permutation of $\omega$. Then we can define the permutation $\hat{\pi}:[0,4^n)\to[0,4^n)$ via removing $\pi^{-1}([0,4^n))\cap [4^n,\infty)$ and shifting everything down to fill the gaps. I.e.,
    \[\hat{\pi}(i)=\pi(i)-|[0,\pi(i))\cap\pi([4^n,\infty))|\]
    Then for any $s$,
    \[\rho_s(\hat{\pi}(X))\geq\rho_s(\pi(X))\]
    since $\hat{\pi}(X)$ contains elements which are pairwise no larger than elements of $\pi(X)$, so the number of them which are smaller than a given $s$ is no larger. Thus a non-big block exists for $\pi$ whenever there is a non-big block for $\hat{\pi}$.
    \\
    \\
    We now describe the construction of $X$: At stage $0$, choose a sub-block $B_0$ of some block such that $\rho_s(\pi_0(B_0))\leq\frac{1}{2}$ for all $s$, which exists by Lemma \ref{lemma:finblock} since for $\hat{\pi_0}$ as described in the previous paragraph, the number of sub-blocks in the $n$-th block is $O(2^n)$, but the number of big sub-blocks is $O(n^2)$. Thus there will eventually be more sub-blocks than big sub-blocks in the $n_0$-th block, and there will be a sub-block which is not big for $\hat{\pi_0}$, and therefore also not big for $\pi_0$ by the previous paragraph. Choose the first such sub-block $B_0$, and set $B_0=X_0$. Set our restraint $\sigma_0=\max(\pi_0^{-1}([0,b_{n_0+1}))+1$.
    \\
    \\
    At stage $s+1$, search for a sub-block $B_{s+1}$ of the $n_{s+1}$-th block which satisfies the following:
    \begin{itemize}
        \item $\pi_i([b_{n_{s+1}},b_{n_{s+1}+1})>\sigma_s$ for all $0\leq i\leq s+1$, i.e. none of the first $s+2$ computable permutations map anything in the $n_{s+1}$-th block below our restraint from the previous stage.
        \item $\rho_t(\pi_i(X_{s}\cup B_{s+1}))\leq\frac{1}{s+2}$ for all $t>\sigma_s+1$
    \end{itemize}
    Such a block and sub-block exist due to Lemma \ref{lemma:finblock}: since the number of sub-blocks of the $n$-th block which are big for a single permutation is $O(n^2)$, the number of sub-blocks which are big for some permutation in a list of $s+2$ permutations is larger by at most a factor of $s+2$: the worst case is when each permutation has a disjoint collection of sub-blocks which are big. However, the number of sub-blocks of the $n$-th block is $2^n$, so eventually there is some $n$ for which there is such a sub-block as this grows much faster than $(s+2)O(n^2)=O(n^2)$. Set $X_{s+1}=X_s\cup B_{s+1}$ and \[\sigma_{s+1}=\max_{i\leq s+1}(\max(\pi_i^{-1}([0,b_{n_{s+1}+1})))+1\]
    Proceed to the next stage.
    \\
    \\
    Finally, let $X=\bigcup_{s\in\omega} X_s$. First, notice that $X$ satisfies the requirements of Lemma \ref{lemma:infblock}, as at each stage a sub-block is added to $X$. Furthermore, $\rho_t(\pi_i(X))\leq\frac{1}{s+2}$ whenever $\sigma_s<t\leq\sigma_{s+1}$ for all $i\leq s+1$, so $\overline{\rho}(\pi_i(X))=0$ for all $i$. Thus $X$ is intrinsically small as desired.
\end{proof}

\section{Weak-adaptability is not stronger than disorderliness} \label{sec:weak-not-stronger-than-disorder}
By strengthening our construction from Theorem~\ref{thm:disorder-stronger-order}, we show that 0-weakly stochastic does not imply intrinsic smallness:
\begin{theorem} \label{thm:weak-not-stronger-than-disorder}
    A host $\A$ that beats all weakly-adaptive, orderly contestants $\g\leq_{\T}\emptyset$ may not beat all disorderly, non-adaptive contestants $\h\leq_{\T}\emptyset$. Furthermore, we shall construct such an $A$ that is not Church stochastic.
\end{theorem}

The $\h$ that $\A$ loses to is the same one from Equation~\eqref{eq:concat-h-all}, which is a concatenation of increasingly nested disordered blocks. To get $\A$, we construct $\A_{s}\supseteq \A_{s-1}$ at stage $s$ so that
\[\A=\bigcup_{s\in\omega} \A_s.\]
We ensure that $\A_s$ is correct up to the $\n_s$-th disordered block $\DB^{\n_s}$:
\[\A\restriction[0,\ceil{\DB^{\n_s}}] =\A_s\restriction[0,\ceil{\DB^{\n_s}}],\]
where $\n_s$ can be obtained using $\emptyset'$ as an oracle.
\begin{lemma}
    We can ensure in Theorem~\ref{thm:weak-not-stronger-than-disorder} that $\A\leq_{\T}\emptyset'$.
\end{lemma}

Our opponents $\g\leq_{\T}\emptyset$ are the adaptive and orderly contestants given by skip rules. Each is a total function $\g:S \to\omega$, where $S\subseteq (\omega\times 2)^{<\omega}$ is the set of skip sequences. We use $\g(X, x)$ to represent the door selected by following the skip sequence $\g$ on $X$ for $x$-many doors.

In our construction, we shall use an enumeration of partial computable functions which contains but is not limited to the skip rules. We shall instead have functions from $2^{<\omega}\times\omega$ to $\omega$, where functions are included in the construction up to the point where they are observed to violate one of the properties of skip rules.\\

Since skip rules are \emph{orderly}, the next door $\g$ chooses $\g(X,x+1)$ should always exceed the previous one $\g(X,x)$:
\[(\forall X\subset\omega)(\forall x\in\omega)\; [\g(X,x+1)>\g(X,x)],\]
$\g$ should also be \emph{consistent} in the sense that next door chosen $\g(X,x)$ should depend only on the contents $X(\g(X,\restriction x))$ of the earlier doors $\g(X,\restriction x)$, and not on the doors it did not open:
\[(\forall X,Y\subset\omega)(\forall x\in\omega)\; [X(\g(X,\restriction x))=Y(\g(X,\restriction x))] \implies \g(X,x)=\g(Y,x).\]

In our construction, it is convenient to abuse notation and to think of $\g$ as a function
\[\g:2^{<\omega}\to2^{\omega}\]
instead, with
\[\g(X) :=\{\g(X,x): x\in\omega\} \subseteq\omega\]
representing the set of all doors opened by $\g$ if the doors that contain cars are given by $X$.\\
 
Fix an effective enumeration $\g_0,\g_1,\ldots$ of the partial-computable functions which will include all these contestants that $\A$ needs to beat. Note that if $\g_i$ is partial or disorderly or inconsistent, $\emptyset'$ will eventually find that out, allowing us to ignore $\g_i$ for the rest of the construction. The framework of $\A$'s construction is similar to that of Section~\ref{sec:DB-all}. At stage $s$, we work within a disordered block that is sufficiently far away, say the $\n_s$-th one $\DB^{\n_s}$. We consider the outcomes of contestants
\[\F:=\{\g_0,\ldots,\g_s\}\]
in $\DB^{\n_s}$, and fill the doors of this disordered block to get $\A_s\restriction \DB^{\n_s}$.

\subsection{Outcomes}
At stage $s$, we work only within $\DB=\DB^{\n_s}$. Like introduced in Section~\ref{sec:strategy1}, we keep track of the restrictions on cars $\cars$ and goat $\goats$ by adding new elements to these sets. At the beginning of the stage, the starting restrictions can be obtained from the previous stage: $\cars=A_{s-1}$, $\goats=[0,\floor{\DB})-\cars$ (Section~\ref{sec:overall}, step~\ref{s:goats}).\\

Under these restrictions, we need to check if a given contestant $\g$ opens doors \emph{densely} or \emph{sparsely} within $\DB$. Using the same ideas as in Section~\ref{sec:outcome1}'s Algorithm $\oc$, we want $\g$'s outcome to be sparse if $\g$ opens fewer than $s$ many doors within some sub-block $\SB\subblock\DB$. Otherwise, $\g$'s outcome should be dense. Refer to Algorithm~\ref{alg:outcome} $\outcome$ for a formalization, which builds on the ideas from $\oc$.\\

The new complication is that $\g$ opens different doors depending on what they have observed so far. When testing for possible outcomes, we obey the following two rules:
\begin{itemize}
    \item Doors $d$ in $\IF\subblock\DB$ are allowed to contain cars (unless $d\in\goats$) or goats.
    \item Doors $d$ in each $\SB\subblock\DB$ contain goats.
\end{itemize}

As we shall see, we can ensure in our construction that while testing outcomes, the doors $d$ above will not already be forced to contain cars $(d\notin\cars)$. These two rules mean we only need to test no more than $2^{|\IF-\goats|}$ possibilities of new car locations $X\subseteq\IF-\goats$ ($\outcome$, line~\ref{l:X}). Of these possibilities, if some $X$ results in $\g$ opening less than $s$ many doors at some $\SB\subblock\DB$ (line~\ref{l:few-doors}), we say that \emph{$\g$ is sparse in $\DB$ at $(X,\SB)$}. Otherwise, we say that the \emph{$\g$ is is dense in $\DB$}. Thus in Algorithm~\ref{alg:outcome}:
\begin{align*}
    \mathcal{X} =\outcome_{\cars,\goats,s}(\g,\DB),
\end{align*}
the output is the set of witnesses $\mathcal{X}$ (line~\ref{l:witnesses}), where $(X,\SB)\in\mathcal{X}$ means that $\g$ is sparse in $\DB$ at this witness. Then $\g$ has a sparse outcome if and only if $\mathcal{X}$ is non-empty.

\begin{algorithm} \label{alg:outcome}
    \caption{$\mathcal{X}=\outcome_{\cars,\goats,s}(\g,\DB)$\\ \Comment{Finds witnesses $(X,\SB)$ for $\g$'s sparseness, if any}}
    \begin{algorithmic}[1]
        \State $\IF\leftarrow$ the increasing part of $\DB$
        \State $\mathcal{X}\leftarrow\emptyset$  \Comment{Set of witnesses for sparseness}
        \For{\textbf{each} $X\subseteq\IF-\goats$} \Comment{Car locations} \label{l:X}
            \For{\textbf{each} $\SB\subblock\DB$}
                \If{$|g(\cars\cup X)\restriction\SB|<s$} \Comment{Filling $X$ with cars, $\g$ will open $<s$ doors in $\SB$} \label{l:few-doors}
                    \State $\mathcal{X}\leftarrow \mathcal{X}\cup\{(X,\SB)\}$ \Comment{Add new witness} \label{l:witnesses}
                \EndIf
            \EndFor
        \EndFor
        \State \Return $\mathcal{X}$
    \end{algorithmic}
\end{algorithm}

\subsection{Strategy}
How should $\A\restriction\DB$ be constructed given the different possible outcomes from a set of contestants $\F$ within $\DB$? We formalize the strategy in Algorithm~\ref{alg:strategy}:
\[(\cars,\goats) =\cg_s(\DB,\F,\cars,\goats),\]
which updates the restriction on cars $\cars$ and goats $\goats$ within the doors of $\DB$, allowing us to later update $\A_s$ to equal $\cars$ (Section~\ref{sec:overall}, step~\ref{s:A}).\\

If we only have a single opponent $\g$, consider the strategy if $\g$ is sparse in $\DB$, say at $(X,\SB)$. We will want to put cars into $X$ (Algorithm $\cg$, line~\ref{l:sparse-cars}) to hold this sparseness. For the same reason, we will forbid cars from entering any sub-block $\SB'\subblock\DB$ before $\SB$, from doors in $\IF$ that are not in $X$ (first union of line~\ref{l:sparse-goats}), and also from the next $s$ many doors opened by $\g$ starting at $\SB$ (second union of line~\ref{l:sparse-goats}):
\[\prin_{\g(\cars)-[0,\floor{\SB})}\restriction[0,s).\]

Note that some of the doors from the second union may go beyond $\DB$.\\

Now $\g$ may receive cars if they open doors in $X$. But we can keep the proportion received low if witness $(X,\SB)$ is \emph{minimal} in the following sense:
\begin{itemize}
    \item $\SB$ is the earliest sub-block that any contestant can be sparse at (line~\ref{l:minimal-sb})
    \item $|X|$ is minimal amongst all contestants that can be sparse at $\SB$ (line~\ref{l:minimal-X}).
\end{itemize}

We refer to these two as the \emph{minimality conditions} for the witnesses. As we shall elaborate in Section~\ref{sec:verify-adaptive}, minimality will ensure that prior to $\SB$, doors within earlier sub-blocks, which we have just filled with goats (first union of line~\ref{l:sparse-goats}), are opened densely enough to dilute the effect of receiving cars at $X$, even when we consider multiple contestants later. To help $\h$ receive enough cars, as we will verify later, we will keep recursing into $\SB$ (line~\ref{l:recurse}) to fill doors that are still empty.\\

On the other hand, if $\g$ has a dense outcome within $\DB$ (line~\ref{l:dense}), we will want to put the cars only at the doors of $\IF\subblock\DB$, as long as not forbidden by $\goats$ (line~\ref{l:dense-cars}). The \ref{eq:large-condition} of $\IF$ will help $\h$ receive enough cars through the $\IF-\goats$ doors. The denseness of $\g$'s outcome means that no matter how the doors in $\IF-\goats$ are filled, $\g$ would open many doors within each $\SB\subblock\DB$, as long as those sub-blocks are filled with goats. Therefore, by forcing doors outside $\IF$ to contain goats (line~\ref{l:dense-goats}), we can ensure that $\g$ receives a low proportion of cars.\\

Next, consider how we can handle multiple contestants $\F$. We use the same nesting strategy as in Section~\ref{sec:construct-dbr}, where we work within a $\DB$ of nestedness level at least $|\F|$. First, we check the outcome of each contestant. If some contestant has a sparse outcome, we say that \emph{$\F$ is sparse in $\DB$} (line~\ref{l:sparse}), otherwise we say that \emph{$\F$ is dense in $\DB$} (line~\ref{l:dense}). If sparse, we find the minimal witness $(X,\SB)$ (lines~\ref{l:minimal-sb},\ref{l:minimal-X}). After updating restrictions on cars (line~\ref{l:sparse-cars}) and goats (line~\ref{l:sparse-goats}) like described earlier for sparse outcomes, we let $\B$ be the set of contestants that are sparse at $(X,\SB)$ (line~\ref{l:B}). As we shall see in Section~\ref{sec:verify-adaptive}, contestants in $\B$ will no longer be affected by future additions of cars within $\DB$, so we can remove $\B$ from consideration and consider the outcomes of $\F-\B$ within $\SB$ only (line~\ref{l:recurse}).\\

We repeat this process recursively, working with fewer contestants with every recursion and zooming into a sub-block of the previous one to fill its doors. After recursing no more than $|\F|$ many times, we will reach an inner-level disordered block where all remaining contestants, if any, are dense within that disordered block (line~\ref{l:dense}). We update the restrictions from a dense outcome (line~\ref{l:dense-cars}-\ref{l:dense-goats}) and terminate the recursion (line~\ref{l:terminate}), which completes the construction of $\A\restriction\DB$.

\begin{algorithm} \label{alg:strategy}
\caption{$(\cars,\goats) =\cg_s(\DB,\F,\cars,\goats)$\\ \Comment{Fill $\DB$'s doors given opponents $\F$}}
\begin{algorithmic}[1]
    \small
    \State $\mathcal{X}\leftarrow \bigcup_{\g\in\F} \{\mathcal{Y}: (\sparse,\mathcal{Y}) =\outcome_{\cars,\goats,s}(\g,\DB)\}$ \Comment{All sparse witnesses}
    \If{$\mathcal{X}\neq\emptyset$} \Comment{Sparse outcome for some $\g\in\F$} \label{l:sparse}
        \State $\SB\leftarrow \mu(\SB)[(\exists X) (X,\SB)\in\mathcal{X}]$ \label{l:minimal-sb} \Comment{Minimal $\SB$}
        \State $X\leftarrow \mu(X)[(X,\SB)\in\mathcal{X}]$
        \State $X\leftarrow X\cap[0,\floor{\SB})$ \label{l:minimal-X} \Comment{Minimal $X$ given $\SB$}
        \State $\cars\leftarrow \cars\cup X$ \label{l:sparse-cars} \Comment{Update cars}
        \State $\goats\leftarrow \goats\cup(\DB-\SB-\cars)\cup\bigcup_{\g\in\F} \prin_{\g(\cars)-[0,\floor{\SB})}\restriction[0,s)$ \label{l:sparse-goats} \Comment{Update goats}
        \State $\B\leftarrow\{\g\in\F: \g \text{ is sparse at } (X,\SB)\}$ \label{l:B} \Comment{\small{$\B$ is unaffected by recursion}}
        \State \Return $\cg_s(\SB,\F-\B,\cars,\goats)$ \label{l:recurse} \Comment{Remove $\B$ and recurse into $\SB$}
    \Else \Comment{Dense outcome for all $\g\in\F$, or $\F=\emptyset$} \label{l:dense}
        \State $\IF\leftarrow$ the increasing part of $\DB$
        \State $\cars\leftarrow \cars\cup\IF-\goats$ \label{l:dense-cars} \Comment{Update cars}
        \State $\goats\leftarrow \goats\cup\DB-\cars$ \label{l:dense-goats} \Comment{Update goats}
        \State \Return $\cars,\goats$ \Comment{Terminate recursion} \label{l:terminate}
    \EndIf
\end{algorithmic}
\end{algorithm}

\subsection{Constructing $\A$} \label{sec:overall}
Initialize $\A=\cars=\goats=\F=\emptyset$. At stage $s$:
\begin{enumerate}
    \item Consider one more contestant:
    \[\F\leftarrow \F\cup\{\g_s\}.\]
    \item Choose a $\DB$ that is sufficiently far away: \label{s:n}
    \begin{align*}
        \n_s\leftarrow &(\mu\n)[\floor{\DB^n}>\max(\cars,\goats)],\\
        \DB\leftarrow &\DB^{n_s}.
    \end{align*}
    \item Pad doors before the new $\DB$ with goats: \label{s:goats}
    \[\goats\leftarrow[0,\floor{\DB})-\cars.\]
    \item Update restrictions on cars $\cars$ and goats $\goats$: \label{s:update}
    \[(\cars,\goats) \leftarrow \cg_s(\DB,\F,\cars,\goats).\]
    
    $\emptyset'$ will eventually know which $\g\in\F$ is partial or disorderly or inconsistent; we can remove these contestants from future consideration.
    \item Add the new cars from $\DB$ into $A$: \label{s:A}
    \[A_s\leftarrow \cars.\]
    \item Go to stage $s+1$.
\end{enumerate}

\subsection{Verification} \label{sec:verify-adaptive}
Given a total computable orderly contestant $\g=\g_r$, we need to satisfy $\Pos_{\g,s}$ for all $s>r$. It is enough to show that:
\begin{lemma}
    Given arbitrary $t\geq s$, let $\DB$ be the $t$-th disordered block. Then for every car that $\g$ receives in $\DB$, they will next receive at least $t$ many consecutive goats.
\end{lemma}
\begin{proof}
    If $t\neq\n_{s'}$ for any $s'$, then $\DB$ will be padded with goats (Section~\ref{sec:overall}, step~\ref{s:goats}), and the lemma holds vacuously. So we may assume $\DB$ is the $\n_{s'}$-th disordered block for some $s'$. Let $\F$ be the set of contestants at stage $t$. Note that $\g\in\F$ since $t\geq r$. At step~\ref{s:update} of the $\A$'s construction, we would update $\cars$ and $\goats$ to include restrictions on cars and goats in $\DB$, calling the recursive $\cg$ algorithm.\\
    
    First consider a recursion depth before $\g$ is removed from the contestant list. $\g$ might receive a car from opening a door $d\in X$. Before passing $\SB$, $\g$ would open more than $t$ many doors in the sub-block $\SB'<\SB$ after $d$, due to the minimality of $\SB$ (line~\ref{l:minimal-sb}) and thus denseness of $\g$ at $\SB'$. But we are forcing doors in $\SB'$ to contain goats (first union in line~\ref{l:sparse-goats}), which does not contradict the shallower recursion depths, since earlier recursions do not add cars to $\SB'$ (line~\ref{l:sparse-cars}). If $\g$ opens the door just before $\SB$, they might again receive a car. But we are also forcing the next $t$ many doors to contain goats (second union of line~\ref{l:sparse-goats}). Even if some of these doors go beyond $\DB$, we will pad those further disordered blocks with goats (section~\ref{sec:overall}, step~\ref{s:goats}) to be consistent with this strategy. Therefore, the effect of gaining one car is always sufficiently diluted with enough goats. After passing $\SB$, if $\g$ still opens doors within $\DB$, they would not receive any cars from the first union again, and from the fact that $X$ will not contain cars past $\SB$ due $X$'s minimality (line~\ref{l:minimal-X}). Within $\SB$ itself, we may assume the lemma holds from induction on recursion depth, since shallower recursions never add cars to the disordered blocks of deeper recursions (line~\ref{l:sparse-cars}).\\
    
    Now consider the case where we are at the recursion depth that will remove $\g$ from consideration (lines~\ref{l:B},\ref{l:recurse}). $\g$ must be sparse at the $(X,\SB)$ of this depth, so within $\SB$, after forcing the next $t$ many doors to contain goats (second union of line~\ref{l:sparse-goats} and Section~\ref{sec:overall} step~\ref{s:goats}), $\g$ will not open any doors in $\SB$ due to sparseness. Deeper recursions (line~\ref{l:recurse}) will only add cars to $\SB$, and will therefore not give $\g$ additional cars. Outside of $\SB$, the lemma holds for $\g$ for the same reasons that it holds in the earlier paragraph.\\
    
    At recursion depths after $\g$ is removed ($\g\in\B$ in line~\ref{l:recurse}), $\g$ will not visit the doors that are yet to be filled, so the lemma holds vacuously.\\

    Finally, consider the situation where $\g$ survives without removal to the terminating recursion depth (line~\ref{l:dense}). Then $\g$ must be dense within the $\DB'$ of this depth. Cars are added to $\DB'$ only at line~\ref{l:dense-cars}, to $\IF'\ll\DB'$, since lower depth recursions never add cars to $\DB'$ (line~\ref{l:sparse-goats}). Now $\DB'$ must contain sub-blocks $\SB''\subblock\DB'$, because the terminating depth cannot exceed the nestedness level of $\DB$, which was picked to exceed the number of contestants. Cars are never put into $\SB''$ (neither from lower depths at line~\ref{l:sparse-cars} nor from the terminating depth at line~\ref{l:dense-cars}). The denseness of $\g$'s outcome must then mean that every possible car gained from opening a door in $\IF'$ will be diluted with at least $t$ many goats in the sub-block after the car.
\end{proof}

Next, we show that given arbitrary $s\in\omega$:
\begin{lemma}
    $\G_s$ holds.
\end{lemma}
\begin{proof}
    We shall show that enough cars are put into $\DB=\DB^{\n_s}$ to satisfy the lemma. Cars are added into $\DB$ at stage $s$ of Section~\ref{sec:overall} during step~\ref{s:update}, when we run the recursive algorithm $\cg$ that puts cars into $\cars$, which will then be put into $\DB$ at step~\ref{s:A}. At the beginning of the stage, no doors in $\DB$ are forced to contain goats since we chose $\DB$ to be far away (step~\ref{s:n}).\\
    
    We shall show that prior to reaching the terminating recursion depth (Algorithm~\ref{alg:strategy}, line~\ref{l:terminate}), the number of doors forced to contain goats by the lower depths can be bounded (line~\ref{l:sparse-goats}). Then, just before terminating, the number of cars added (line~\ref{l:dense-cars}) will exceed the bound sufficiently to satisfy the lemma.\\
    
    Let $\DB'$ be the inner-level disordered block we are working within at the terminating depth. Notice that at lower recursion depths, goat restrictions that apply to $\DB'$ appear only at the second union of line~\ref{l:sparse-goats}, where we force the first $s$ many doors opened by each contestant to contain goats. Now there are no more than $s$ many contestants when we work within $\DB$, bounding the number of such doors by $s^2$. Also, since we remove at least one contestant before recursing (line~\ref{l:recurse}), the maximum depth of recursion is bounded by $s$. So no more than $s^3$ many doors within $\DB'$ will be forced to contain goats.\\
    
    When we add cars to $\DB'$ just prior to terminating, the number of cars added ($\IF-\goats$ in line~\ref{l:dense-cars}) will be plentiful from the \ref{eq:large-condition}, which can tolerate up to $\n_s^3\geq s^3$ many goats and still satisfy $G_{\n_s}$, and thus $G_s$.
\end{proof}

\subsection{Not Church Stochastic}
Note that this is not Church stochastic for the same reason that the counterexample constructed in Theorem \ref{thm:disorder-stronger-order} was not: the adaptive strategy which checks all doors in order, selecting the remainder of any level 0 sub-block in which it finds a car, will select a set of high density on $A$. (The reason this fails for weak stochasticity is that a weakly adaptive strategy must include every bit that it checks, thereby gaining no advantage from checking every bit in search of 1's.)
\section{Future work} \label{sec:future}

The main open question is the missing implication from Figure \ref{fig:zoo}, which is a formalization of whether adaptability can beat disorderliness:

\begin{openquestion}
    Is every MWC stochastic set also intrinsically dense?
\end{openquestion}

The natural strategy to prove this question is to prove that MWC stochastic sets are closed under computable permutations like random sets. However, Merkle \cite{merkle_2003} proved that the MWC stochastic sets are not closed under computable permutations. This does not completely resolve the question, as the image of the witnessing MWC stochastic set under the chosen computable permutation still has the same density, so this does not witness a failure of intrinsic density.
\\
\par
Using the techniques from \cite{thesis}, we can prove that the Turing degrees of sets which have computable density but not intrinsic density are closed upwards. Astor \cite{intrinsiccontent} proved that the Turing degrees of the intrinsically dense sets are exactly the High or DNC degrees, however it is unknown if the same can be said of these degrees.

\begin{openquestion}
    Does every High or DNC degree compute a set with computable density which is not intrinsically dense? Are the Turing degrees of the computably dense sets exactly the High or DNC degrees?
\end{openquestion}

Answering the second question positively would mean that orderly strategies are just as complex as disorderly ones in the degree-theoretic sense. When separating intrinsic density and computable density, as well as weak stochasticity and intrinsic density, we showed that the notions did not coincide for any $\density$. However, our separation of weak stochasticity and Church stochasticity only works for stochasticity $0$.

\begin{openquestion}
    Does weak stochasticity coincide with Church stochasticity for any $\density$?
\end{openquestion}

One attempt to answer this would be to find a notion of randomness which separates them. Schnorr randomness is an attractive candidate, however it is unknown if every Schnorr random is weakly stochastic.
\\
\par
If the answer to this question is yes, that would be the first known instance where the stochasticity zoo changes based on $\density$.
\\
\par
We can define \emph{weak KL stochasticity} to be the natural stochasticity notion which is disorderly and weakly adaptive: doors can be selected out of order, and previously seen information can be used to make decisions. However, the contestant must always take the object hidden behind each door they select. Then essentially the same proof as for Theorem \ref{thm:weak-non-church} proves that this notion is strictly weaker than KL stochasticity and does not imply MWC stochasticity. 

\begin{openquestion}
    Does weak Church (or weak MWC) stochasticity imply weak KL stochasticity? Does MWC stochasticity imply weak KL stochasticity?
\end{openquestion}

%

\bibliographystyle{plain}
\bibliography{references}

\end{document}